\documentclass[12pt]{amsart}

\usepackage{amsfonts, amsmath, amssymb, amsthm}
\usepackage{changebar}
\usepackage{xypic}

\newtheorem{lemma}{Lemma}
\newtheorem{theorem}{Theorem}

\newtheorem{proposition}{Proposition}
\newtheorem{corollary}{Corollary}
\newtheorem{definition}{Definition}

\newtheorem{remark}{Remark}

\newtheorem{example}{Example}

\newcommand{\nth}[1]{$#1 {\rm - th }$}

\newcommand{\rad}{{\rm rad } }
\newcommand{\Rad}{{\rm Rad } }
\newcommand{\sexp}{{\rm sexp} }
\newcommand{\ord}{{\rm ord} }

\newcommand{\rk}{{\rm rank } \ }
\newcommand{\zprk}{\Z_p\hbox{-rk}}

\newcommand{\Z}{\mathbb{Z}}
\newcommand{\ZM}[1]{\Z /( #1 \cdot \Z)}

\newcommand{\diag}{{\bf Diag}}

\newcommand{\rg}[1]{\mbox{\bf #1}}
\newcommand{\eu}[1]{\mathfrak{#1}}
\newcommand{\id}[1]{\mathcal{#1}}
\newcommand{\Gal}{\mbox{ Gal }}
\newcommand{\Gals}{\hbox{\tiny Gal}}
\newcommand{\Ker}{\mbox{ Ker }}

\newcommand{\prk}{p{\rm -rk } }

\newcommand{\rf}[1]{(\ref{#1})}

\newcommand{\Norm}{\mbox{\bf N}}
 
\newcommand{\lchooses}[2]{\left( \frac{#1}{#2 } \right)} 
 
\newcommand{\Mat}{\hbox{Mat}}

\newcommand{\F}{\mathbb{F}}

\newcommand{\U}{\mathbb{U}}
\newcommand{\B}{\mathbb{B}}
\newcommand{\K}{\mathbb{K}}

\newcommand{\KH}{\mathbb{H}}

\newcommand{\KL}{\mathbb{L}}

\newcommand{\M}{\mathbb{M}}
\newcommand{\T}{\mathbb{T}}

\newcommand{\Q}{\mathbb{Q}}

\newcommand{\C}{\mathbb{C}}
\newcommand{\N}{\mathbb{N}}
\newcommand{\pn}{^{1/p^{n}}}

\def\ra{\rightarrow}

\def\tj{\tilde{\jmath}}
\def\tt{\tilde{\tau}}

 \newcommand{\ran}{\rangle}
\newcommand{\lan}{\langle}

\begin{document}
{\obeylines \small
\vspace*{-1.0cm}
\hspace*{3.5cm}One by one the guests arrive
\hspace*{3.5cm}The guests are coming through
\hspace*{3.5cm}And ``Welcome, welcome'' cries a voice
\hspace*{3.5cm}``Let all my guests come in!''\footnote{Leonard Cohen: {\em The Guests}.}.
\vspace*{0.4cm}
\hspace*{5.0cm} {\it To S. J. Patterson, at his $60$ - th birthday}
\vspace*{0.5cm}
\smallskip
}
\title[SNOQIT I] {SNOQIT I\footnote{SNOQIT=``Seminar Notes on Open Questions
in Iwasawa Theory'' refers to a seminar held together with J.S.Patterson in
2007/08.}: Growth of $\Lambda$-modules and Kummer theory} \author{Preda
Mih\u{a}ilescu} \address[P. Mih\u{a}ilescu]{Mathematisches Institut der
Universit\"at G\"ottingen} \email[P. Mih\u{a}ilescu]{preda@uni-math.gwdg.de}
\thanks{The first results of this paper were developed before the year $2008$,
while the author was sponsored by a grant of the Volkswagen Foundation. Some
of the central questions were first considered during the month of February
$2010$, spent as a guest professor in Caen, at the invitation of Bruno
Angl\`es: these led also to the more recent developments.}

\date{Version 2.10 \today}
\vspace{1.0cm}
\begin{abstract}
Let $A = \varprojlim_n$ be the projective limit of the $p$-parts of the class
groups in some $\Z_p$-cyclotomic extension. The main purpose of this paper is
to investigate the transition $\Lambda a_n \ra \Lambda a_{n+1}$ for some
special $a = (a_n)_{n \in \N} \in A$, of infinite order. Using an analysis of
the $\F_p[ T ]$-modules $\id{A}_n/p \id{A}_n$ and $\id{A}_n[ p ]$ we deduce
some restrictive conditions on the structure and rank of these modules.  Our
model can be applied also to a broader variety of cyclic $p$-extensions and
associated modules. In particular, it applies to certain cases of subfields of
Hilbert or Takagi class fields, i.e. finite cyclic extensions.

As a consequence of this taxonomy (The term of \textit{taxonometric} research
was coined by Samuel Patterson; it very well applies to this work and is part
of the dedication at the occasion of his 60-th birthday) we can give a proof
in CM fields of the conjecture of Gross concerning the non vanishing of the
$p$-adic regulator of $p$-units.
\end{abstract}
\maketitle
\tableofcontents

\section{Introduction}
Let $p$ be an odd prime and $\K \supset \Q[ \zeta ]$ be a galois extension
containing the \nth{p} roots of unity, while $(\K_n)_{n \in \N}$ are the
intermediate fields of its cyclotomic $\Z_p$-extension $\K_{\infty}$. Let $A_n
= (\id{C}(\K_n))_p$ be the $p$-parts of the ideal class groups of $\K_n$ and
$\rg{A} = \varprojlim_n A_n$ be their projective limit. The subgroups
$\rg{B}_n \subset A_n$ are generated by the classes containing ramified primes
above $p$ and we let 
\begin{eqnarray}
\label{bram}
A'_n & = & A_n/\rg{B}_n, \\
\rg{B} & = & \varprojlim_n \rg{B}_n, \quad \rg{A}' = \rg{A}/\rg{B}.\nonumber
\end{eqnarray}
We denote as usual the galois group $\Gamma = \Gal(\K_{\infty}/\K)$ and
$\Lambda = \Z_p[ \Gamma ] \cong \Z_p[[ \tau ]] \cong \Z_p[[ T ]]$, where $\tau
\in \Gamma$ is a topological generator and $T = \tau-1$; we let
\[ \omega_n = (T+1)^{p^{n-1}} - 1 \in \Lambda, \quad 
\nu_{n+1,n} = \omega_{n+1}/\omega_n \in \Lambda. \]
If $X$ is a finite abelian group, we denote by $X_p$ its $p$ - Sylow
group. The exponent of $X_p$ is the smallest power of $p$ that annihilates
$X_p$; the \textit{subexponent}
\[ \sexp(X_p) = \min \{ \ \ord(x) \ : \ x \in X_p \setminus X_p^p \ \} . \]

Fukuda proves in \cite{Fu} (see also Lemma \ref{fukuda} below) that if
$\mu(\K) = 0$, then there for the least $n_0 \geq 0$ such that
$\prk(A_{n_0+1}) = \prk(A_{n_0})$ we also have $\prk(\rg{A}) = prk(A_{n_0})$:
the $p$-rank of $A_n$ becomes stationary after the first occurance of a
stationary rank. It is a general property of finitely generated
$\Lambda$-modules of finite $p$-rank, that their $p$-rank must become
stationary after some fixed level -- the additional fact that this already
happens after the first rank stabilization is a consequence of an early
theorem of Iwasawa (see Theorem \ref{iw6} below) which relates the
$\Lambda$-module $\rg{A}$ to class field theory. The theorem has a class field
theoretical proof and one can show that the properties it reveals are not
shared by arbitrary finitely generated $\Lambda$-modules.

The purpose of this paper is to pursue Fukuda's observation at the level of
individual cyclic $\Lambda$-modules and also investigate the
\textit{prestable} segment of these modules. We do this under some
simplifying conditions and focus on specific cyclic $\Lambda$-modules defined
as follows:
\begin{definition}
\label{dconic}
Let $\K$ be a CM field and $a = (a_n)_{n \in \N} \in
\rg{A}^-$ have infinite order. 
We say that $a$ is {\em conic}\footnote{One can easily provide
examples of non conic elements, by considering $\K_{\infty}$ as a
$\Z_p$-extension of $\K_n$ for some $n > 1$. It is an interesting question to
find some conditions related only to the field $\K$, which assure the
existence of conic elements.} if the following conditions are fulfilled:
\begin{itemize}
\item[ 1. ] There is a $\Lambda$-submodule $\rg{C} \subset \rg{A}^-$ such that 
\[ \rg{A}^- = \rg{C} \oplus \Lambda a. \]
We say in this case that $\Lambda a$ is $\Lambda$-complementable.
\item[ 2. ] Let $c = (c_n)_{n \in \N} \in \Lambda a$. If $c_n = 1$ for some $n
> 0$, then $c \in \omega_{n} (\Lambda a)$.  
\item[ 3. ] If $b \in \rg{A}^-$ and there is a power $q = p^k$ with $b^q \in
\Lambda a$, then $b \in \Lambda a$.
\item[ 4. ] If $f_a(T) \in \Z_p[ T ]$ is the exact annihilator of $\Lambda a$,
  then $(f_a(T), \omega_n(T)) = 1$ for all $n > 0$.
\end{itemize}
\end{definition}
The above definition is slightly redundant, containing all the properties that
we shall require. See also \S 2.1 for a more detailed discussion of the
definition. 

The first purpose of this paper is to prove the following theorem.
\begin{theorem}
\label{main}
Let $p$ be an odd prime, $\K$ be a galois CM extension containing a
\nth{p} root of units and let $ \K_n, A_n$ and $\rg{A}$ be defined
like above. Let $a = (a_n)_{n \in \N} \in \rg{A}^- \setminus
(\rg{A}^-)^p$ be conic, $q = \ord(a_1)$ and let $f_a(T)$ be the exact
annihilator polynomial of $a$. If $q = p$, then $f_a(T)$ is an
Eisenstein polynomial. Otherwise, if $n_0$ is the least integer with
$\prk(\id{A}_{n_0}) = \prk(\id{A}_{n_0+1})$ then either $n_0 \leq 3$
and the rank is bounded by
\begin{eqnarray}
\label{bound}
v_p(a_1) > 1 \quad \Rightarrow \quad \prk(\Lambda a) < p(p-1),
\end{eqnarray}
or $n_0 > 3$ and $\sexp(\id{A}_{n_0-1}) = \exp(\id{A}_{n_0-1})$.
Moreover $\id{A}_{n_0}$ has an annihilator polynomial 
\[ f_{n_0}(T) = T^{\lambda} - q w(T), \quad \lambda = \prk(\id{A}), \quad w(T)
\in (\Z_p[ T ])^{\times}, \]
and $f_a(T) - f_{n_0}(T) \in a_{n_0}^{\top} = \{ x \in Z_p[ T ] : a_{n_0}^x =
1\}$.
\end{theorem}

The theorem is obtained by a tedious algebraic analysis of the rank
growth in the \textit{transitions} $\id{A}_n \hookrightarrow
\id{A}_{n+1}$. 

A class of examples of conic modules is encountered for quadratic ground
fields $\K$, such that the $p$-part $A_1(\K)$ of the class group is
$\Z_p$-cyclic. We shall give in section 3.2 a series of such examples, drawn
from the computations of Ernvall and Mets\"ankil\"a in \cite{EM}. A further
series of applications concern the structure of the components $e_{p-2k}
\rg{A}$ of the class group of $p$-cyclotomic extensions, when the Bernoulli
number $B_{2k} \equiv 0 \bmod p$. If the conjecture of Kummer - Vandiver or
the cyclicity conjecture holds for this component, then the respective modules
are conic.

The question about the detailed structure of annihilator polynomials in
Iwasawa extensions is a difficult one and it has been investigated in a series
of papers in the literature. For small, e.g. quadratic fields, a probabilistic
approach yields already satisfactory results. In this respect, the
Cohen-Lenstra \cite{CL} and Cohen-Martinet \cite{CM} heuristics have imposed
themselves, being confirmed by a large amount of empirical results; see also
Bhargava's use of these heuristics in \cite{Ba} for recent developments.

At the other end, for instance in $p$-cyclotomic fields, computations only
revealed linear annihilator polynomials. In spite of the improved resources of
modern computers, it is probably still infeasible to pursue intensive numeric
investigations for larger base fields. In this respect, we understand the
present paper as a proposal for a new, intermediate approach between empirical
computations and general proofs: {\em empirical case distinctions} leading to
some structural evidence. In this sense, the conditions on conic elements are
chosen such that some structural results can be achieved with feasible
effort. The results indicate that for large base fields, the repartitions of
exact annihilators of elements of $\rg{A}^-$ can be expected to be quite
structured and far from uniform repartition within all possible distinguished
polynomials.
\subsection{Notations}
We shall fix some notations. The field $\K$ is assumed to be a CM galois
extension of $\Q$ with group $\Delta$, containing a \nth{p} root of unity
$\zeta$ but no \nth{p^2} roots of unity. We let $(\zeta_{p^n})_{n \in \N}$ be
a norm coherent sequence of \nth{p^n} roots of unity, so $\K_n = \K[
\zeta_{p^n} ]$. Thus we shall number the intermediate extensions of
$\K_{\infty}$ by $\K_1 = \K, \K_n = \K[ \zeta_{p^n} ]$. We have uniformly that $\K_n$ contains the \nth{p^n} but not the
\nth{p^{n+1}} roots of unity. In our numbering, $\omega_n$ annihilates
$\K_n^{\times}$ and all the groups related to $\K_n$ ($A_n, \id{O}(\K_n)$,
etc.)

Let $A = \id{C}(\K)_p$, the $p$ - Sylow subgroup of the class group
$\id{C}(\K)$. The $p$-parts of the class groups of $\K_n$ are denoted by $A_n$
and they form a projective sequence with respect to the norms $N_{m,n} :=
\Norm_{\K_m/\K_n}, m > n > 0$, which are assumed to be surjective. The
projective limit is denoted by $\rg{A} = \varprojlim_n A_n$. The submodule
$\rg{B} \subset \rg{A}$ is defined by \rf{bram} and $\rg{A}' =
\rg{A}/\rg{B}$. At finite levels $A'_n = A_n/\rg{B}_n$ is isomorphic to the
ideal class group of the ring of the $p$-units in $\K_n$. The maximal
$p$-abelian unramified extension of $\K_n$ is $\KH_n$ and $\KH'_n \subset
\KH_n$ is the maximal subfield that splits all the primes above $p$. Then
$\Gal(\KH'_n/\K_n) \cong A'_n$ (e.g. \cite{Iw}, \S 3. - 4.).

If the coherent sequence $a = (a_n)_{n \in \N} \in \rg{A}^-$ is a conic
element, then $\prk(\Lambda a) < \infty$. We write $\id{A} = \Lambda a$ and
$\id{A}_n = \Lambda a_n$: the finite groups $\id{A}_n$ form a projective
sequence of $\Lambda$-modules with respect to the norms. The exact annihilator
polynomial of $\id{A}$ is denoted by $f_a(T) \in \Z_p[ T ]$.

If $f \in \Z_p[ T ]$ is some distinguished polynomial that divides the
characteristic polynomial of $\rg{A}$, we let $\rg{A}(f) = \cup_n \rg{A}[ f^n
]$ be the union of all power $f$-torsions in $\rg{A}$. Since $\rg{A}$ is
finitely generated, this is the maximal submodule annihilated by some power of
$f$. If $B \subset \rg{A}(f)$ is some $\Lambda$-module, then we let $k =
\ord_f(B)$ be the least integer such that $f^k B = 0$.

\subsection{List of symbols}
We give here a list of the notations introduced below in connection with
Iwasawa theory
\begin{eqnarray*}
\begin{array}{l c l}
p & & \hbox{A rational prime}, \\ 
\zeta_{p^n} & & \hbox{Primitive \nth{p^n}
roots of unity with $\zeta_{p^n}^p = \zeta_{p^{n-1}}$ for all $n > 0$.},\\
\mu_{p^n} &  & \{ \zeta_{p^n}^k, k \in \N \}, \\
\K &  & \hbox{ A galois CM extension of $\Q$ containing the \nth{p} roots of unity}, \\
\K_{\infty}, \K_{n}&  & \hbox{The cyclotomic $\Z_p$ - extension of $\K$, and intermediate fields,}\\
\Delta &  & \Gal(\K/\Q), \\
A(\rg{K}) & = & \hbox{$p$-part of the ideal class group of the field $\rg{K}$}, \\
s &  & \hbox{The number of primes above $p$ in $\K$}, \\
\Gamma &  & \Gal( \K_{\infty}/\K ) = \Z_p \tau, \quad \hbox{$\tau$ a
  topological generator of $\Gamma$} \\
T & = & \tau -1,\\ 
* &  & \hbox{Iwasawa's involution on $\Lambda$ induced by $T^* = (p-T)/(T+1)$},\\
\end{array}
\end{eqnarray*}
\begin{eqnarray*}
\begin{array}{l c l}
\Lambda &  & \Z_p[[ T ]], \quad \Lambda_n = \Lambda/(\omega_n \Lambda), \\
\omega_n &  & (T+1)^{p^{n-1}} - 1, \quad (\K_n^{\times})^{\omega_n} = \{ 1 \}, \\
A'_n = A'(\K_n) &  & \hbox{The $p$ - part of the ideal class group of the $p$ - integers of $\K_n$}, \\
A' & = & \varprojlim A'_n, \\
\rg{B} & = & \lan \{ b = (b_n)_{n \in \N} \in A : b_n = [ \wp_n ], \wp_n \supset (p) \} \ran_{\Z_p}, \\
\KH_{\infty} &  & \hbox{The maximal $p$ - abelian unramified extension of $\K_{\infty}$},\\
\KH'_{\infty} \subset \KH_{\infty} &  & \hbox{The maximal subextension of $\KH_{\infty}$ that splits the primes above $p$}.
\end{array}
\end{eqnarray*}
The following notations are specific for transitions:
\begin{eqnarray*}
\begin{array}{l c l}
(A, B) & = & \hbox{A conic transition, $A, B$ are finite $\Z_p[ T ]$-modules}, \\
G & = & < \tau > , \hbox{a cyclic $p$-group acting on the modules of the transition}, \\
T & = &  \tau - 1, \\
S(X) & = & X[ p ], \ \hbox{the $p$-torsion of the $p$ group $X$, or its socle}, \\
R(X) & = & X/(p X), \ \hbox{the ``roof'' of the $p$ group $X$}, \\
N, \iota & = & \hbox{The norm and the lift associated to the transition $(A,B)$}, \\
K & = & \Ker(N : B \ra A), \\
\omega & = & \hbox{Annihilator of $A$, such that $N = p + \omega N'$}, \\
d & = & \deg(\omega(T)); \quad \nu = \frac{(\omega+1)^p-1}{\omega}, \\
\nu \omega & = & \hbox{Annihilator of $B$}, \\
\id{T} & = & B/\iota(A), \ \hbox{The transition module associated to $(A, B)$}, \\
s, s' & = & \hbox{Generators of $S(A), S(B)$ as $\F_p[ T ]$-modules}, \\
a, b & = & \hbox{Generators of $A, B$ as $\Z_p[ T ]$-modules}, \\
r, r' & = & \prk(A), \prk(B).
\end{array}
\end{eqnarray*}

\subsection{Ramification and its applications}
Iwasawa's Theorem 6 \cite{Iw} plays a central role in our investigations. Let
us recall the statement of this theorem in our context (see also \cite{Wa},
Lemma 13.14 and 13.15 and \cite{La}, Chapter 5, Theorem 4.2):
\begin{theorem}[Iwasawa, Theorem 6 \cite{Iw}]
\label{iw6}
Let $\K$ be a number field and $P = \{\wp_i : i = 1, 2, \ldots, s\}$ be the
primes of $\K$ above $p$ and assume that they ramify completely in
$\K_{\infty}/\K$. Let $\KH_{\infty}/\K_{\infty}$ be the maximal $p$-abelian
unramified extension of $\K_{\infty}$ and $H =
\Gal(\KH_{\infty}/\K)$, while $I_i \subset H, i = 1, 2, \ldots, s$
are the inertia groups of some primes of $\KH_{\infty}$ above $\wp_i$. Let
$a_i \in \rg{A}$ be such that $\varphi(\sigma_i) I_1 = I_i$, for $i = 2,
3, \ldots, s$.

For every $n > 0$ there is a module $Y_n = \omega_n \rg{A} \cdot
\left[a_2, \ldots, a_s \right]_{\Z_p}$ such that $\rg{A}/Y_n \cong
\rg{A}_n$.
\end{theorem}
%
Note that the context of the theorem is not restricted to
CM extensions. In fact Iwasawa's theorem applies also to non cyclotomic
$\Z_p$-extensions, but we shall not consider such extensions in this paper.

The following Theorem settles the question about $Y_1^-$ in CM extensions:
\begin{theorem}
\label{simram}
Let $\K$ be a galois CM extension of $\Q$ and $\rg{A}$ be defined
like above. Then $\rg{A}^-(T) = \rg{A}^-[ T ] = \rg{B}^-$.
\end{theorem}
We prove the theorem in Chapter 4. Then we derive from Theorem \ref{simram}
the following result:
\begin{corollary}
\label{gross}
Let $\K$ be a CM $s$-field and $\rg{B}_n, A'_n$ be defined by \rf{bram}. Then
$(\rg{A}')^-[ T ] = \{ 1 \}$.
\end{corollary}

This confirms a conjecture of Gross and Kuz'min stated by Federer and Gross in
\cite{FG} in the context of $p$-adic regulators of $p$-units of number fields,
and earlier by Kuz'min \cite{Ku} in a class field oriented statement, which
was shown by Federer and Gross to be equivalent to the non vanishing of
$p$-adic regulators of $p$-units. We prove here the class field theoretic
statement for the case of CM fields. The conjecture was known to be true for
abelian extensions, due to previous work of Greenberg \cite{Gr1}.

\subsection{Sketch of the proof}
We start with an overview of the proof. Our approach is based on the
investigation of the growth of the ranks $r_n := \prk(\id{A}_n) \ra
\prk(\id{A}_{n+1})$; for this we use \textit{transitions} $C_n :=
\id{A}_{n-1}/\iota_{n,n+1}(\id{A}_n)$, taking advantage of the fact that our
assumptions assure that the ideal lift maps are injective for all $n$. Since
we also assumed $\prk(\id{A}) < \infty$, it is an elementary fact that the
ranks $r_n$ must stabilize for sufficiently large $n$. Fukuda proved recently
that this happens after the first $n$ for which $r_n = r_{n+1}$. We call this
value $n_0$: the {\em stabilization index}, and focus upon the
\textit{critical section } $\id{A}_n : n < n_0$. In this respect, the present
work is inspired by Fukuda's result and extends it with the investigation of
the critical section; this reveals useful criteria for stabilization, which
make that the growth of conic $\Lambda$-modules is quite controlled: at the
exception of some modules with \textit{flat} critical section, which can grow
in rank indefinitely, but have constant exponent $p^k$, the rank is bounded by $p(p-1)$.

The idea of our approach consists in modeling the \textit{transitions} $(A,B)
= (\id{A}_n, \id{A}_{n+1})$ by a set of dedicated properties that are derived
from the properties of conic elements. The conic transitions are introduced in
Definition \ref{dcontrans} below. Conic transitions do not only well describe
the critical section of conic $\Lambda$-modules, but they also apply to
sequences $A_1, A_2, \ldots, A_n$ of more general finite modules on which a
$p$-cyclic group $< \tau >$ acts via the group ring $\rg{R} = \ZM{p^N}[ \tau ]
= \Z_p[ T ], T = \tau - 1$, with $p^N A_n = 0$; the modules $A,B$ are in
particular assumed to be cyclic as $\rg{R}$-modules and they fulfill some
additional properties with respect to norms and lifts. As a consequence, the
same theory can be applied for instance to sequences of class groups in cyclic
$p$-extensions, ramified or unramified.

The ring $\rg{R}$ is a local ring with maximal ideal $(p, T)$; since
this ideal is not principal, it is customary to use the Fitting ideals
for the investigation of modules on which $\rg{R}$ acts. Under the
additional conditions of conicity however, the transitions $(A, B)$
come equipped with a wealth of useful $\F_p[ T ]$-modules. Since
$\F_p[ T ]$ has a maximal ideal $(T)$, which is principal, this highly
simplifies the investigation. The most important $\F_p[ T ]$-modules
related to a transition are the \textit{socle}, $S(B) = B[ p ]$ and
the \textit{roof}, $R(B) = B/p B$. It is a fundamental, but not evident
fact, that $S(B)$ is a cyclic $\F_p[ T ]$, and we prove this by
induction in Lemma \ref{cycsoc}. With this, the transitions are caught
between two pairs of cyclic $\F_p[ T ]$-modules, and the relation
between these modules induces obstructions on the growth types. These
obstructions are revealed in a long sequence of tedious case
distinctions, which develop in a natural way. 

The relation between rank growth and norm coherence reveals in Corollary
\ref{termt} the principal condition for termination of the rank growth:
assuming that $\prk(A_1) = 1$, this must happen as soon as $\prk(A_n) <
p^{n-1}$. This is a simple extension of Fukuda's results, giving a condition
for growth termination, not only for rank stabilization. A further important
module associated to the transition is the kernel of the norm, $K := \Ker(N :
B \ra A)$. The structure of $K$ is an axiom of conic transitions, which is proved
to hold in the case of conic $\Lambda$-modules. The analysis of growth in
conic transitions is completed in the Chapter 2.
 
In Chapter 3, the analysis of transitions can be easily adapted to conic
$\Lambda$-modules, yielding an inductive proof of their structure, as
described in Theorem \ref{main}. In the fourth chapter we prove the Theorem \ref{simram} and Corollary \ref{gross}.

Except for the second Chapter, the material of this paper is quite simple and
straight forward. In particular, the main proof included in Chapter 3 follows
easily from the technical preparation in Chapter 2. Therefore the reader
wishing to obtain first an overview of the main ideas may skip the second
chapter in a first round and may even start with Chapter 4, in case her
interest goes mainly in the direction of the proof of the conjectures included
in that chapter.

The Lemmata \ref{fukuda}, \ref{ab} are crucial for our approach to
Kummer theory. They imply the existence of some index $n_0 \geq 0$
such that for all coherent sequences $a = (a_n)_{n \in \N} \in \rg{A}$
of infinite order, there is a constant $z = z(a) \in \Z$ such that:
\begin{eqnarray}
\label{stab}
\prk(\id{A}_n) & = & \prk(\id{A}_{n_0}), \nonumber \\
a_{n+1}^p & = & \iota_{n,n+1}(a_n) \\
\ord(a_n) & = & p^{n+z}. \nonumber
\end{eqnarray}

\section{Growth of $\Lambda$-modules}
We start with a discussion of the definition of conicity:
\subsection{The notion of conic modules and elements}
We have chosen in this paper a defensive set of properties for conic modules,
in order to simplify our analysis of the growth of $\Lambda$-modules. We give
here a brief discussion of these choices. The restriction to CM fields and
submodules $\Lambda a \subset \rg{A}^-$ is a sufficient condition for ensuring
that all lift maps $\iota_{n,n+1}$ are injective. One can prove in general
that for $a$ of infinite order, these maps are injective beyond a fixed
stabilization index $n_0$ that will be introduced below. For $n < n_0$ the
question remains still open, if it suffices to assume that $\ord(a) = \infty$
in order to achieve injectivity at all levels. It is conceivable that the
combination of the methods developed in this paper may achieve this goal, but
the question allows no simple answer, so we defer it to latter investigations.

By assuming additionally that $(f_a(T), \omega_n(T)) = 1$, we obtain as a
consequence of these assumptions, that for $x = (x_m)_{m \in \N}$ with $x_m =
1$, we have $x \in \omega_n \rg{A}$. In the same vein, if $x_m^{\omega_n} = 1$
for $m > n$, then $x_m \in \iota_{n,m}(\id{A}_n)$. These two consequences are
very practical and will be repeatedly applied below.

The fundamental requirement to conic elements $a \in \rg{A}$, is that the
module $\Lambda a$ has a direct complement which is also a $\Lambda$-module.
Conic modules exist -- see for instance Corrolary 1 or the case of imaginary
quadratic extensions $\K$ with $\Z_p$ cyclic $(\id{C}(\K))_p$ and only one
prime above $p$. The simplifying assumption allows to derive interesting
properties of the growth of $\Lambda$-modules, that may be generalized to
arbitrary modules.

This condition in fact implies the property 3. of the definition
\ref{dconic}, a condition which we also call
\textit{$\Z_p$-coalescence closure} of $\Lambda a$, meaning that
$\Lambda a$ is equal to the smallest $\Z_p$-submodule of $\rg{A}$,
which contains $\Lambda a$ and has a direct complement as a
$\Z_p$-module. Certainly, given property 1 and using additive notation,
$b = g(T) a + x, x \in \rg{B}, g \in \Z_p[ T ]$, and then $q b \in \Lambda
a$ implies by property 1 that $q x = 0$, so $b$ is twisted by a
$p$-torsion element, which is inconsistent with the fact that
$\rg{A}^-$ was assumed $\Z_p$-torsion free. It is also an interesting
question, whether the assumption of property 1 and $a_i \in T
\rg{A}$ are sufficient to imply property 1.

\subsection{Auxiliary identities and lemmata}
We shall frequently use some identities in group rings, which are
grouped below. For $n > 0$ we let $\rg{R}_n = \ZM{p^N}[ T
]/(\omega_n)$ for some large $N > 0$, satisfying $N > \exp(A_n)$. The
ring $\rg{R}_n$ is local with maximal ideal $(\omega_n)$ and we write
$\overline{T}$ for the image of $T$ in this ring. Since
$\overline{T}^{p^n} \in p \rg{R}_n$, it follows that
$\overline{T}^{p^{n+N}} = 0$; thus $\overline{T} \in \rg{R}_n$ is
nilpotent and $\rg{R}_n$ is a principal ideal domain.

We also consider the group ring $\rg{R}'_n:=\ZM{p^N}[ \omega_n
]/(\nu_{n+1,n})$, which is likewise a local principal ideal domain with
maximal ideal generated by the nilpotent element $\omega_n$. From the binomial
development of $\nu_{n+1,n}$ we deduce the following fundamental identities in
$\Lambda$
\begin{eqnarray}
\label{norm} 
\nu_{n+1,n} & = & \frac{(\omega_n+1)^p-1}{\omega_n} = \omega_n^{p-1} +
\sum_{i=1}^{p-1} \binom{p}{i}/p \cdot \omega_n^{i-1} \nonumber \\ & = & p (1 +
O(\omega_n)) + \omega_n^{p-1} = \omega_n^{p-1} + p u(\omega_n), \\ & &
u(\omega_n) = 1 + \frac{p-1}{2} \omega_n + \ldots + \omega_n^{p-2} \in
\Lambda^{\times}, \nonumber \\ \omega_{n}^p & = & \omega_n \cdot (\nu_{n+1,n}
- p u(\omega_n)) = \omega_{n+1} - p \omega_n u(\omega_n). \nonumber
\end{eqnarray}

The above identities are equivariant under the Iwasawa involution $* :
\tau \mapsto (p+1) \tau^{-1}$.  Note that we fixed the cyclotomic
character $\chi(\tau) = p+1$. If $f(T) \in \Lambda$, we write $f^*(T)
= f(T^*)$, the \textit{reflected} image of $f(T)$. The reflected norms
are $\nu^*_{n+1,n} = \omega_{n+1}^*/\omega_n^*$. From the definition
of $\omega_n^*$ we have the following useful identity:
\begin{eqnarray}
\label{om*om} 
\omega_n + t \omega_n^* = p^{n-1} c, \quad t \in \Lambda_n^{\times}, c \in
\Z_p^{\times}.
\end{eqnarray} 

We shall investigate of the growth of the modules $A_n$ for $n \ra
n+1$. Suppose now that $A$ is a finite abelian $p$-group which is cyclic as
an $\Z_p[ T ]$-module, generated by $a \in A$. We say that a monic polynomial
$f \in \Z_p[ T ]$ is a \textit{minimal polynomial} for $a$, if $f$ has minimal
degree among all monic polynomials $g \in a^{\top}= \{ x \in \Z_p[ T ] : x a =
0 \} \subset \Z_p[ T ]$.
 
We note the following consequence of Weierstrass preparation:
\begin{lemma}
\label{minp}
Let $I = (g(T)) \subset \Z_p[ T ]$ be an ideal generated by a monic polynomial $g(T) \in \Z_p[ T ]$. If $n = \deg(g)$ is minimal amongst all the degrees of
monic polynomials generating $I$, then $g(T) = T^n + p h(T)$, with $h(T) \in
\Z_p[ T ]$ and $\deg(h) < n$.
\end{lemma}
\begin{proof}
Let $g(T) = T^n + \sum_{i=0}^{n-1} c_i T^i$. Suppose that there is some $i <
n$ such that $p \nmid c_i$. Then the Weierstrass Preparation Theorem
(\cite{Wa}, Theorem 7.3) implies that $g(T) \Z_p[ T ] = g_2(T) \Z_p[ T ]$, for
some polynomial with $\deg(g_2(T)) \leq n$, which contradicts the choice of
$g$. Therefore $p \mid c_i$ for all $0 \leq i < n$, which completes the proof
of the lemma.
\end{proof}
\begin{remark}
\label{anpol}
As a consequence, if $A$ is a finite abelian group which is a $\Lambda$-cyclic
module of $p$-rank $n$, then there is some polynomial $g(T) = T^r - p h(T)$
which annihilates $A$.
\end{remark}

We shall use the following simple application of Nakayama's Lemma:
\begin{lemma}
\label{naka}
Let $X$ be a finite abelian $p$-group of $p$-rank $r$ and $\id{X} = \{ x_1,
x_2, \ldots, x_r \} \subset X$ be a system with the property that the images
$\overline{x}_i \in X/ p X$ form a base of this $\F_p$-vector space. Then
$\id{X}$ is a system of generators of $X$.
\end{lemma}
\begin{proof}
This is a direct consequence of Nakayama's Lemma, \cite{La1}, Chapter VI, \S
6, Lemma 6.3.
\end{proof}

The following auxiliary lemma refers to elementary abelian $p$ groups with
group actions.
\begin{lemma}
\label{trank}
Let $E$ be an additively written finite abelian\footnote{These groups are
sometimes denoted by {\em elementary abelian $p$-groups},
e.g. \cite{Wa},\S 10.2.} $p$-group of exponent $p$. Suppose there is a cyclic
group $G = < \tau >$ of order $p$ acting on $E$, and let $T = \tau-1$. Then
$E$ is an $\F_p[ T ]$-module and $E/T E$ is an $\F_p$-vector space. It $r =
\dim_{\F_p}(E/ T E)$, then every system $\id{E} = \{ e_1, e_2, \ldots, e_r\}
\subset E$ such that the images $\overline{e}_i \in E/(T E)$ form a base of
the latter vector space, is a minimal system of generators of $E$ as an $\F_p[
T ]$-module.  Moreover $E[ T ] \cong E/(T E )$ as $\F_p$-vector spaces and $E
= \oplus_{i=1}^r \F_p[ T ] e_i$ is a direct sum of $r$ cyclic $\F_p[ T
]$-modules.
\end{lemma}
\begin{proof}
The modules $E[ T ]$ and $E/ T E$ are by definition annihilated by $T$; since 
$\F_p[ T ]/(T \F_p[ T ]) \equiv \F_p$, they are finite dimensional 
$\F_p$-vector spaces. Let $\id{E}$ be defined like in the hypothesis. The 
ring $\F_p[ T ]$ is local with principal maximal ideal $T \F_p[ T ]$, and $T$ 
is a nilpotent of the ring since $\tau^p = 1$ so we have the following 
identities in $\F_p[ \tau ] = \F_p[ T ]$: $0 = \tau^p - 1 = (T+1)^p-1 = T^p$. 
It follows from Nakayama's Lemma, that $\id{E}$ is a minimal system of 
generators. The map $T : E \ra E$ is a nilpotent linear endomorphism of the 
$\F_p$-vector space $E$, so the structure theorem for Jordan normal forms of 
nilpotent maps implies that $E = \oplus_{i=1}^r \F_p[ T ] e_i$ 
One may also read this result by considering the exact sequence
\begin{eqnarray*}
\xymatrix{ 0 \ar@{->}[r] & E[ T ] \ar@{->}[r] & E \ar@{->}[r] & E
\ar@{->}[r] & E/(T E) \ar@{->}[r] & 0 }
\end{eqnarray*}
in which the arrow $E \ra E$ is the map $e \mapsto T e$. 
The diagram indicates that $E[ T ] \cong E/(T E)$, hence the claim.
\end{proof}

In the situation of Lemma \ref{trank}, we denote the common $\F_p$-dimension
of $E[ T ]$ and $E/T E$ by {\em $T$-rank of $E$}.
\subsection{Stabilization}
We shall prove in this section the relations \rf{stab}. First we introduce the
following notations: 
\begin{definition}
\label{socroof}
given a finite abelian $p$-group $X$, we write $S(X) = X[ p ]$ for its
$p$-torsion: we denote this torsion also by {\em the socle} of
$X$. Moreover the factor $X/X^p = R(X)$ -- the {\em roof} of $X$. Then
$S(X)$ and $R(X)$ are $\F_p$-vector spaces and we have the classical
definition of the $p$-rank given by $\prk(X) = \rk(S(X)) = \rk(R(X))$, the
last two ranks being dimensions of $\F_p$-vector spaces. 
We say that $x \in
X$ is {\em $p$-maximal}, or simply maximal, if $x \not \in X^p$.

Suppose there is a cyclic $p$-group $G = \tau$ acting on $X$, such that $X$ is
a cyclic $\Z_p[ T ]$-module with generator $x \in X$, where $T =
\tau-1$. Suppose additionally that $S(X)$ is also a cyclic $\F_p[ T
]$-module. Let $s:= (\ord(x)/p) x \in S(X)$. Then we say that $S$ is {\em
straight} if $s$ generates $S(X)$ as an $\F_p[ T ]$-module; otherwise, $S(X)$
is {\em folded}.
\end{definition}
The next lemma is a special case of Fukuda's Theorem 1 in \cite{Fu}:
\begin{lemma}[Fukuda]
\label{fukuda}
Let $\K$ be a CM field and $A_n, \rg{A}$ be defined like
above. Suppose that $\mu(\rg{A}^-) = 0$ and there is an $n_0 > 0$,
such that $\prk(A^-_{n_0}) = \prk(A^-_{n_0+1})$. Then $\prk(A^-_n) =
\prk(A^-_{n_0}) = \lambda^-$ for all $n > n_0$.
\end{lemma}
\begin{remark}
  The above application of Fukuda's Theorem requires $\mu = 0$; it is
  known that in this case the $p$-rank of $A_n$ must stabilize, but
  here it is shown that it must stabilize after the first time this
  rank stops growing from $A_n$ to $A_{n+1}$. We have restricted the
  result to the minus part which is of interest in our context. Note
  that the condition $\mu = 0$ can be easily eliminated, by
  considering the module $(\rg{A}^-)^{p^m}$ for some $m > \mu$.
\end{remark}

The following elementary, technical lemma will allow us
to draw additional information from Lemma \ref{fukuda}. 
\begin{lemma}
\label{ab}
Let $A$ and $B$ be finitely generated abelian $p-$groups denoted additively,
and let $N:B\ra A$, $\iota:A\ra B$ two $\Z_p$ - linear maps such that:
\begin{itemize}
\item[1.] $N$ is surjective and $\iota$ is injective\footnote{The same results
can be proved if the injectivity assumption is replaced by the assumption that
$\sexp(A) > p$ -- injectivity then follows. In our context we injectivity is 
however part of the premises, so we give here the proof of the simpler variant 
of the lemma};
\item[2.] The $p-$ranks of $A$ and $B$ are both equal to $r$ and $| B |/| A |
= p^r$.
\item[3.]  $N(\iota(a))=p a,\forall a\in A$ and $\iota$ is rank preserving, 
so $\prk(\iota(A)) = \prk(A)$;
\end{itemize}
Then $\iota(A) = p B$ and $\ord(x) = p \cdot \ord( N x)$ for all $x \in B$.
\end{lemma}
\begin{proof}
The condition 3. is certainly fulfilled when $\iota$ is injective, as we did, but it also follows from $\sexp(A) > p$, even for lift maps that are not injective.
  We start by noting that for any finite abelian $p$ - group $A$ of
  $p$ - rank $r$ and any pair $\alpha_i, \beta_i; \ i = 1, 2, \ldots,
  r$ of minimal systems of generators there is a matrix $E \in \Mat(r,
  \Z_p)$ which is invertible over $\Z_p$, such that
\begin{eqnarray}
 \label{unimod}
\vec{\beta} = E \vec{\alpha}.
\end{eqnarray}
This can be verified directly by
extending the map $\alpha_i \mapsto \beta_i$ linearly to $A$ and,
since $(\beta_i)_{i=1}^r$ is also a minimal system of generators,
deducing that the map is invertible, thus regular. It represents a
unimodular change of base in the vector space $A \otimes_{\Z_p} \Q_p$.

The maps $\iota$ and $N$ induce maps
\[ {\overline \iota}: A/pA \ra B/pB,\;{\overline N}:B/pB\ra A/pA.\]
From 1, we see ${\overline N}$ is surjective and since, by 2., it is a
map between finite sets of the same cardinality, it is actually an
isomorphism. But 3. implies that ${\overline N} \circ {\overline
  \iota}:A/pA\ra A/pA$ is the trivial map and since ${\overline N}$ is
an isomorphism, ${\overline \iota}$ must be the trivial map, hence
$\iota(A) \subset pB$.

Since $\iota$ is injective, it is rank preserving, i.e. $\prk(A) =
\prk(\iota(A))$.  Let $b_i, \ i = 1, 2, \ldots, r$ be a minimal set of
generators of $B$: thus the images $\overline{b}_i$ of $b_i$ in $B/p B$ form
an $\F_p$ - base of this algebra. Let $a_i = N(b_i)$; since $\prk(B/p B) =
\prk(A/p A)$, the set $(a_i)_i$ also forms a minimal set of generators for
$A$. We claim that $| B/\iota(A) | = p^r$.

Pending the proof of this equality, we show that $\iota(A) = p B$. 
Indeed, we have the equality of $p$- ranks:
\[ | B/ p B | = | A/ p A | = | B/\iota(A) | = p^r, \] implying that $|
p B | = | \iota(A) |$; since $\iota(A) \subset p B$ and the $p$ -
ranks are equal, the two groups are equal, which is the first
claim. The second claim will be proved after showing that $|
B/\iota(A) | = p^r$.

Let $S(X)$ denote the socle of the finite abelian $p$ - group
$X$. There is the obvious inclusion $S(\iota(A)) \subset S(B) \subset
B$ and since $\iota$ is rank preserving, $\prk(A) = \prk(S(A)) =
\prk(B) = \prk(S(B)) = \prk(S(\iota(A)))$, thus $ S( B ) = S(\iota(A)
) $. Let $(a_i)_{i=1}^r$ be a minimal set of generators for $A$ and
$a'_i = \iota(a_i) \in B, i = 1, 2, \ldots, r$; the $(a'_i)_{i=1}^r$
form a minimal set of generators for $\iota(A) \subset B$. We choose
in $B$ two systems of generators in relation to $a'_i$ and the matrix
$E$ will map these systems according to \rf{unimod}.

First, let $b_i \in B$ be such that $p^{e_i} b_i = a'_i$ and $e_i > 0$
is maximal among all possible choices of $b_i$. From the equality of
socles and $p$ - ranks, one verifies that the set $(b_i)_{i=1}^r$
spans $B$ as a $\Z_p$-module; moreover, $\iota(A) \subset p B$ implies
$e_i \geq 1$. On the other hand, the norm being surjective, there is a
minimal set of generators $b'_i \in B, \ i = 1, 2, \ldots, r$ such
that $N(b'_i) = a_i$. Since $b_i, b'_i$ span the same finite
$\Z_p$-module $B$, \rf{unimod} in which $\vec{\alpha} = \vec{b}$ and
$\vec{\beta} = \vec{b'}$ defines a matrix with $\vec{b} = E \cdot
\vec{b'}$.  On the other hand,
\[ \iota(\vec{a}) = \vec{a'} = \diag(p^{e_i}) \vec{b} =
\diag(p_i^{e_i}) E \cdot \vec{b'}, \] 

The linear map $N : B \rightarrow A$ acts component-wise on vectors
$\vec{x} \in B^r$. Therefore,
\begin{eqnarray*}
  N \vec{b} & = & \vec{N b_i} = N (E \vec{b'}) = N \left( (\prod_j
    {b'}_j^{\sum_j e_{i,j}})_{i=1}^r \right) \\ & = & \left( \prod_j (N
    b'_j)^{\sum_j e_{i,j}} \right)_{i=1}^r = \left( \prod_j (a_j)^{\sum_j
    e_{i,j}} \right)_{i=1}^r \\ & = & E (\vec{a}).
\end{eqnarray*}
Using the fact that the subexponent is not $p$, we obtain thus two expressions
for $N \vec{a'}$ as follows:
\begin{eqnarray*}
  \vec{N a'} & = & p \vec{a} = p I \cdot \vec{a} \\ & = & N\left(\diag(p^{e_i})
    \vec{b}\right) = \diag(p^{e_i}) \cdot N (\vec{b}) = \diag(p^{e_i}) \cdot E
  \vec{a}, \quad \hbox{so } \\ \vec{a} & = & \diag(p^{e_i -1 }) \cdot E
  \vec{a}
\end{eqnarray*}
The $a_j$ form a minimal system of generators and $E$ is regular over $\Z_p$;
therefore $\vec(\alpha) := (\alpha_j)_{j=1}^r = E \vec{a}$ is also minimal
system of generators of $A$ and the last identity above becomes
\[ \vec{a} = \diag(p^{e_i -1 }) \cdot \vec{\alpha}.\] If $e_i > 1$ for
some $i \leq r$, then the right hand side is not a generating system of $A$
while the left side is: it follows that $e_i = 1$ for all $i$.
Therefore $| B /\iota(A) | = p^R$ and we have shown above that this
implies the injectivity of $\iota$.

Finally, let $x \in B$ and $q = \ord(N x) \geq p$. Then $q N (x) = 1 = N( q x
)$, and since $q x \in \iota(A)$, it follows that $N(q x) = p q x = 1$ and
thus $p q$ annihilates $x$. Conversely, if $\ord(x) = p q$, then $p q x = 1 =
N( q x ) = q N(x)$, and $\ord(N x) = q$. Thus $\ord(x) = p \cdot \ord(N x)$
for all $x \in B$ with $\ord(x) > p$. If $\ord( x ) = p$, then $x \in S(B) =
S(\iota(A) \subset \iota(A)$ and $N x = p x = 1$, so the last claim holds in
general.
\end{proof}

One may identify the modules $A, B$ in the lemma with subsequent levels
$A_n^-$, thus obtaining:
\begin{proposition}
\label{sf}
Let $\K$ be a CM field, let $\rg{A}^- = \varprojlim_n A_n^-$ and assume that
$\mu(\rg{A}) = 0$.  Let $n_0 \in \N$ be the bound proved in Lemma
\ref{fukuda}, such that for all $n \geq n_0$ and for all submodules $B \subset
\rg{A}^-$ we have $\prk(B_n) = \zprk(B) = \lambda(B)$.  Then the following
hold:
\begin{eqnarray}
\label{ordinc}
\quad p x & = & \iota( N_{n+1,n}(x) ), \quad \iota(A^-_n) = p A^-_{n+1},
\nonumber \\ \quad \omega_{n} x & \in & \iota_{n,n+1}(A^-_n[ p ])
\end{eqnarray}
In particular, for $n > n_0$
\begin{eqnarray}
\label{ordinc2}
\nu_{n+1,n}(a_{n+1}) = p a_{n+1} = \iota_{n,n+1}(a_n).
\end{eqnarray}
\end{proposition}
\begin{proof} 
We let $n > n_0$. Since $\rg{A}^-$ is $\Z_p$-torsion free, we may also assume
that $\sexp(A^-_n) > p$. We use the notations from Lemma \ref{ab} and let
$\iota = \iota_{n,n+1}, N = N_{n+1,n}$ and $N' = \nu_{n+1,n}$.

For proving \rf{ordinc2}, thus $p x = \iota(N(x)) = N'(x)$, we consider the
  development $ t:=\omega_n = (T+1)^{p^n} - 1$ and
  \[ N' = p + t \cdot v = p + t \left(\binom{p}{2} + t w )\right),
  \quad v, w \in \Z[ t ], \] as follows from the binomial development of $N' =
  \frac{(t+1)^p-1}{t}$. By definition, $t$ annihilates $A^-_n$ and a fortiori
  $\iota(A^-_n) \subset A^-_{n+1}$; therefore, for arbitrary $x \in A^-_{n+1}$
  we have $(p t) x = t (p x) = t \iota(x_1) = 0$, where the existence of $x_1$
  with $p x = \iota(x_1), x_1 \in A^-_n$ follows from Lemma \ref{ab}. Since
  $\iota$ is injective and thus rank preserving, we deduce that $t x \in
  A^-_{n+1}[ p ] = \iota(A^-_n[p])$, which is the first claim in
  \rf{ordinc}. Then
  \[ t^2 x = t \cdot ( t x ) = t x_2 = 0 , \quad \hbox{ since $x_2 = t
    x \in \iota(A^-_n)$}. \] Using $t^2 x = p t x = 0$, the above development
  for $N'$ plainly yields $N' x = p x$, as claimed. Injectivity of the lift
  map then leads to \rf{ordinc}. Indeed, for $a = (a_n)_{n \in \N}$ and $n >
  n_0$ we have
\begin{eqnarray*}
  \ord(a_n) & = & \ord(\iota_{n+1,n}(a_n)) = \ord( \iota_{n+1,n} \circ
  N_{n+1,n} (a_{n+1})) \\ & = & \ord( \nu_{n+1,n} a_{n+1} ) = \ord( p a_{n+1} )
  = \ord(a_{n+1})/p.
\end{eqnarray*}
This completes the proof.
\end{proof}
\begin{remark}
\label{ordgen}
The restriction to the minus part $\rg{A}^-$ is perfectly compatible with the context of this paper. However, we note that Lemma \ref{ab} holds as soon as $\sexp(A) > p$. As a consequence, all the facts in Proposition \ref{sf} 
hold true for arbitrary cyclic modules $\Lambda a$ with $\ord(a) = \infty$. The proof being algebraic, it is not even necessary to assume that $\K_{\infty}$ is the cyclotomic $\Lambda$-extension of $\K$, it may be any $\Z_p$-extension and $\rg{A} = \varprojlim_n A_n$ is defined with respect to the $p$-Sylow groups of the class groups in the intermediate levels of $\K_{\infty}$. The field $\K$ does not need to be CM either. The Proposition \ref{sf} is suited for applications in Kummer theory, and we shall see some in the Chapter 4. This remark shows that the applications reach beyond the frame imposed in this paper.
\end{remark}

As a consequence we have the following elegant description of the growth of
orders of elements in $A_n^-$:
\begin{lemma}
\label{Gabriele}
Let $\K$ be a CM field and $A_n, \rg{A}$ be defined as above, with
$\mu(\rg{A}) = 0$. Then there exists an $n_0 > 0$ which only depends on $\K$,
such that:
\begin{itemize}
\item[ 1. ] $\prk(A^-_n) = \prk(A^-_{n_0}) = \lambda^-$ for $n \geq n_0$,
\item[ 2. ] For all $a = (a_n)_{n \in \N} \in \rg{A}^-$ there is a $z = z(a)
\in \Z$ such that, for all $n \geq n_0$ \rf{stab} holds.
\end{itemize}
\end{lemma}
\begin{proof}
The existence of $n_0$ follows from Lemma \ref{fukuda} and relation
\rf{ordinc} implies that $\ord(a_{n}) = p^{n-n_0} \ord(a_{n_0})$ for all $n
\geq n_0$, hence the definition of $z$. This proves point 2 and \rf{stab}.
\end{proof}
The above identities show that the structure of $\id{A}$ is completely
described by $\id{A}_{n_0}$: both the rank and the annihilator $f_a(T)$ of
$\id{A}$ are equal to rank and annihilator of $\id{A}_{n_0}$. Although
$\id{A}_{n_0}$ is a finite module and thus its annihilator ideal is not
necessarily principal, since it also contains $\omega_{n_0+1}$ and $p^{n+z}$,
the polynomial $f_a(T)$ is a distinguished polynomial of least degree, contained
in this ideal. Its coefficients may be normed by choosing minimal
representatives modulo $p^{n+z}$. It appears that the full information about
$\id{A}$ is contained in the \textit{critical section} $\{ \id{A}_n : n \leq
n_0 \}$.
\subsection{The case of increasing ranks}
In this section we shall give some generic results similar to Lemma
\ref{ab}, for the case when the groups $A$ and $B$ have distinct
ranks. Additionally, we assume that the groups $A$ and $B$ are endowed
with a common group action which is reminiscent from the action of
$\Lambda$ on the groups $\id{A}_n$ of interest.

The assumptions about the groups $A, B$ will be loaded with additional
premises which are related to the case $A = \id{A}_n, B = \id{A}_{n+1}$. We
define:
\begin{definition}
\label{dcontrans}
A pair of finite abelian $p$-groups $A, B$ is called a \textit{conic
transition}, if the following hold:
\begin{itemize}
\item[ 1. ] $A, B$ are abelian $p$-groups written additively and $N: B \ra A$
and $\iota: A \ra B$ are linear maps which are surjective, respectively
injective. Moreover $N \circ \iota = p$ as a map $B \ra B$. The ranks are $r =
\prk(A) \leq r' = \prk(B)$. Note that for $r=r'$ we are in the case of Lemma
\ref{ab}, so this will be considered as a {\em stable} case.
\item[ 2. ] There is a finite cyclic $p$-group $G = \Z_p \tau$ acting on $A$
and $B$, making $B$ into cyclic $\Z_p[ \tau ]$ - modules. We let $T = \tau-1$.
\item[ 3. ] We assume that there is a polynomial $\omega(T) \in
\rg{R}:=\ZM{p^N}[ T ]$, for $N > 2 \exp(B)$, with 
\begin{eqnarray*}
N & = &\frac{(\omega(T)+1)^p-1}{\omega} \in \rg{R}, \\
\omega & \equiv & T^{\deg(\omega)} 
\bmod p \Z_p[ T ], \quad \hbox{and} \quad \omega \equiv 0 \bmod T.
\end{eqnarray*}
In particular, \rf{norm} holds; we write $d = \deg(\omega(T)) \geq 1$.
We also assume that $\omega A = 0$.
\item[ 4. ] The kernel $K := \Ker(N : B \ra A) \subset B$ is assumed to verify
$K = \omega B$ and if $x \in B$ verifies $\omega x = 0$, then $x \in \iota(A)$.
\item[ 5. ] There is an $a \in A$ such that $a_i = T^i a, i = 0, \ldots, r-1$
form a $\Z_p$-base of $A$, and $a_0 = a$.
\end{itemize}
The transition is \textit{regular} if $r' = p d$; it is regular \textit{flat},
if $\sexp(B) = \exp(B)$ and it is regular wild, if it is regular and $\exp(B)
> \sexp(B)$. It is \textit{initial} is $r = d = 1$ and it is \textit{terminal}
if $r' < p d$. If $r = r'$, the transition is called \textit{stable}. The
module associated to the transition $(A,B)$ is the \textit{transition module}
$\id{T} = B/\iota(A)$.  We shall write $\nu = \iota \circ N : B \ra B$. Then
$\nu = \nu(T)$ is a polynomial of degree $\deg(\nu) = (p-1) d$ and $\omega
\nu$ annihilates $B$.

We introduce some notions for the study of socles. Let $\varpi: B \ra \N$ be 
the map $x \ra \ord(x)/p$ and $\psi: B \ra S(B)$ be given by $x \mapsto 
\varpi(x) \cdot x$, a $\Z_p$-linear map. Let $\Omega(b) = \{ q_i := 
\varpi(T^i b) \ : \ i = 0, 1, \ldots, r'-1\}$. Then $q_0 \geq q_1 \geq \ldots 
q_{r'-1}$. The {\em jumps} of $\Omega(b)$ are the set 
\[ J := \{ i \ : \ q_i > q_{i+1} \} \subset \{ 0, 1, \ldots, r'-2\} . \]
We shall write 
\[ B_j := \sum_{i = 0}^{j} \Z_p T^i b \subset B, \quad 0 \leq j < r'. \]
\end{definition}
We consider in the sequel only transitions that are not stable, thus we
assume that $r < r'$. We show below that point 4 of the definition reflects
the specific properties of conic modules, while the remaining ones are of
general nature and apply to transitions in arbitrary cyclic
$\Lambda$-modules. Throughout this chapter, $a$ and $b$ are generators of $A$
and $B$ as $\Z_p[ T ]$-modules. Any other generators differ from $a$ and $b$
by units.

We start with an elementary fact which holds for finite cyclic $\Z_p[ T
]$-modules $X$, such as the elements of conic transitions.
\begin{lemma}
\label{byun} 
If $(A, B)$ be a conic transition. If $y, z \in B \setminus p B$ are such
that $y - z \in p B$, then they differ by a unit:
\begin{eqnarray}
\label{unitdif}
\quad \quad y, z \not \in p B, \quad y-z \in p B \quad \Rightarrow \quad
\exists \ v(T) \in (\Z_p[ T ]^{\times}), \ z = v(T) y.
\end{eqnarray}

Moreover, if $S(A), S(B)$ are $\F_p[ T ]$-cyclic and $y \in B \setminus \{ 0
\}$ is such that $T y \in T S(B)$, then either $y \in S(B)$ or there are $a'
\in \iota(a)$ and $z \in S(B)$ such that
\begin{eqnarray}
\label{presoc}
y = z + a', \quad T a' = 0, \quad \ord(a') > p.
\end{eqnarray}
\end{lemma}
\begin{proof}
Let $b \in B$ generate this cyclic $\Z_p[ T ]$ module. Then $R(B) = B/p B$ is 
a cyclic $\F_p[ T ]$ module; since $\tau^{p^M} = 1$ for some
$M > 0$, it follows that $(T+1)^{p^M} - 1 = T^{p^M} = 0 \in \F_p[ T ]$, so the
element $T$ is nilpotent.

Let $y \in B$ with image $0 \neq \overline{y} \in R(B)$. Then there is a $k
\geq 0$ such that $T^k b \F_p[ T ] R(B) = y \F_p[ T ] R(B)$: consider
the annihilator ideal of the image $b' \in B/(p B, y)$. Since $b$ is a 
generator, we also have $y = g(T) b, g \in \Z_p[ T ]$. The above shows that 
$g(T) \equiv 0 \bmod T^k$, so let $g(T) = T^k g_1(T)$ with $g_1(T) = \sum_{j 
\geq 0} c_j T^j, \ c_j \in \Z_p$. The above equality of ideals in $B / p B$ 
implies that $c_0 \in \Z_p^{\times}$, since otherwise $T^k B \F_p[ T ] R(B) 
\supsetneq y \F_p[ T ] R(B)$. Therefore $g(T) \in (\Z_p[ T ])^{\times}$. 
Applying the same fact to $y, z$, we obtain \rf{unitdif} by transitivity.

Finally suppose that $0 \neq T y \in S(B) = B[ p ]$. Then $y \neq 0$; if
$\ord(y) = p$ then $y \in S(B)$ and we are done. Suppose thus that $\ord(y) =
p^e, e > 1$ and let $y' = p^{e-1} y \in S(B)$. The socle $S(B)$ is $\F_p[ T ]$
cyclic, so there is a $z \in S(B)$ such that $T y = T z \in T S(B)$. Then
$T(y-z) = 0$ and $y-z \in \iota(A)$ by point 4 of the definition
\ref{dcontrans}; therefore $y = z + a', a' \in \iota(a)$. Moreover, $\ord(y) =
\ord(a') > p$ while $T y = T z + T a'$, thus $T a' = 0$. This confirms
\rf{presoc}.
\end{proof}

\subsection{Transition modules and socles}
The following lemmata refer to conic transitions. We start with several 
results of general nature, which will then be used in the next section
for a case by case analysis of transitions and minimal polynomials.
\begin{lemma}
\label{cycsoc}
The following facts hold in conic transitions:
\begin{itemize}
 \item[(i)] Suppose that $S(A)$ is $\F_p[ T ]$-cyclic; then the socle $S(B)$
is also a cyclic $\F_p[ T ]$-cyclic module.
\item[(ii)] Let $x^{\top} = \{ t \in \Z_p[ T ] \ : \ t x = 0 \}$ be the
annihilator ideal of $x$ and $\overline{\omega} \in \Z_p[ T ]$ be a
representant of the class $(\omega \bmod b^{\top}) \in B/b^{\top}$. We have
\begin{eqnarray}
\label{sockern}
\iota(S(A)) & \subset & S(K), \quad K \cap \iota(A) = \iota(S(A)), \quad
\hbox{and} \quad \\
\label{exactk}
K & = & \overline{\omega} B = a^{\top} B. 
\end{eqnarray}
\end{itemize}
\end{lemma}
\begin{proof}
The point (i) follows from Lemma \ref{trank}. Indeed, $S(B) \supseteq 
\iota(S(A))$ are elementary $p$-groups by definition. If $x \in S(B)[ T ]$, 
then $T x = 0$ and point 3 of the definition of conic transitions implies 
that $x \in \iota(A) \cap S(B) = \iota(S(A))$. Thus $S(B)[ T ] \subseteq 
\iota(S(A))[ T ]$ and since $S(A)$ is $\F_p[ T ]$ cyclic, we know that 
$\prk(\iota(S(A))[ T ]) = \prk(S(B)[T]) = 1$. The Lemma \ref{trank} implies 
that the $T$-rank of $S(B)$ is one and $S(B)$ is cyclic as an 
$\F_p[ T ]$-module.

Let now $x \in \iota(S(A))$, so $\omega x = p x = 0$. Then $N x = (p u +
\omega^{p-1}) x = 0$, and thus $x \in S(K)$. If $x' \in \iota(A) \cap K$, then
$N x' = p x' = 0$ and thus $x' \in S(K) \cap \iota(A) = \iota(S(A))$, showing
that \rf{sockern} is true.

By point 4. of the definition of conic transitions, we have $K = \omega B =
\omega \Z_p[ T ] b$ and since $\omega$ acts on $b$ via its image modulo the
annihilator of this generator, it follows that $K = \overline{\omega} b$ for
any representant of this image in $\Z_p[ T ]$. For $t \in a^{\top}$ we have
$N( t b ) = t N(b) = t a = 0$; conversely, if $x = t' b \in K$, then $N (t' b)
= t' N( b ) = t' a = 0$ and thus $t' \in a^{\top}$. We thus have $K = a^{\top}
B$, which confirms \rf{exactk} and completes the proof of (ii).
\end{proof}
An important consequence of the structure of the kernel of the norm is
\begin{corollary}
\label{termt}
Let $(A, B)$ be a transition with $r < d$. Then $r' = r$.
\end{corollary}
\begin{proof}
Let $\theta = T^r + p g(T) \in a^{\top}$ be a minimal annihilator polynomial
of $a$. Then $\theta B \subset K = \overline{\omega} B$ so there is a $y \in \Z_p[ T ]$
such that $\theta b = \omega y b$ and since $\theta b \not \in p B$, it
follows that $y \not \in p \Z_p[ T ]$. Let thus $y = c T^j + O (p, T^{j+1})$
with $j \geq 0$ and $(c,p) = 1$. Then 
\[ \alpha = T^r + p g(T) - \omega \cdot y = T^r + T^{d+j} + O(p, T^{d+j+1}) \in
b^{\top}.\] 
Using $d > r$, Weierstrass Preparation implies that there is a distinguished
polynomial $h(T)$ of degree $r$ and a unit $v(T)$ such that $\alpha = h v$.
Since $v$ is a unit, $h \in b^{\top}$. But then $T^r b \in P$ and thus $B = P$
and $\prk(B) = r$, which confirms the statement of the Lemma.
\end{proof}
The corollary explains the choice of the signification of flat and terminal
transitions: a terminal transition can only be followed by a stable one.

We analyze in the next lemma the transition module $\id{T}$ in detail.
\begin{lemma}
\label{transmod}
Let $(A,B)$ be a conic transition and $\id{T} = B/\iota(A)$ be its transition
module. Then 
\begin{itemize}
 \item[ 1. ] The module $\id{T}$ is $\Z_p[ T ]$-cyclic, annihilated by $\nu$.
 and
\[ r' \leq r + (p-1) d. \]
Moreover
\begin{eqnarray}
\label{exps}
\exp(B) \leq p \exp(A),
\end{eqnarray}
and there is an $\ell(B)$ with $\ord(T^{\ell-1} b) = \exp(B) = p \cdot \ord(
T^{\ell} B)$. 
\item[ 2. ] If $S(A) \subset T A$ and $S(B)$ is folded, then $r' = r$.
\end{itemize}
\end{lemma}
\begin{proof}
Since $\nu(B) = \iota(A)$, it follows that $\nu(\id{T}) = 0$, 
showing that 
\[ r' = \prk(B) \leq \prk(A) + \prk(\id{T}) \leq \prk(A) + \deg(\nu) = r + 
(p-1) d.\]
which confirms the first claim in point 1.

Let now $q = \exp(A)$, so $q \iota(A) = 0$ and $q B \supseteq \iota(A)$. Thus
$\id{T}^{\top} \supseteq (B/q B)$. We let $\ell(B) = \prk(B/q B)$ and prove
the claims of the lemma. We have
\begin{eqnarray*}
p q u(T) b = - q T^{p-1} b + q \iota(a) = - q T^{p-1} b,
\end{eqnarray*}
Assuming that $q T^{p-1} b \neq 0$, we obtain $\prk(\id{T}) \leq (p-1) d <
\prk(B/q B)$, in contradiction with the fact that $B/qB$ is a quotient of
$\id{T}$. Therefore $q T^{p-1} b = 0$ and thus $p q u(T) b = 0$, so $\exp(B) =
p q$. Therefore, the module $q B \subset S(B)$ and it has has rank
$\ell(B)$. Let $s' \in S(B)$ be a generator. Comparing ranks in the $\F_p[ T
]$-cyclic module $S(B)$, we see that $T^{r'-\ell} s = q b v(T), v(T) \in
(\F_p[ T ])^{\times}$.

Suppose now that $S(B)$ is folded; then $\iota(S(A)) = T^k S(B), k = r'-r$. If
$s = T g(T) \iota(a)$ is a generator of $\iota(S(A))$, then $T^{r'-r} s' =
v(T) \iota(s), v(T) \in (\F_p[ T ])^{\times}$. Thus
\[ T ( T^{r'-r-1} s' - g(T) v \iota(a)) = 0, \]
and by point 4 of the definition of conic transitions, $T^{r'-r-1} s' \in 
\iota(S(A))$. But then 
\[ r' = \prk(S(B)) \leq \prk(\iota(S(A)) + r'-(r+1) = r + r' - (r+1) = r'-1.\]
This is a contradiction which implies that $r = r'$ and $(A,B)$ is in this 
case a stable transition.  
\end{proof}

In view of the previous lemma, we shall say that the transition $(A,B)$ is
\textit{wild} if $r' = p d$ and $S(B)$ is folded. The flat transitions are
described by:
\begin{lemma}
\label{lreg}
Let $(A, B)$ be a conic transition. The following conditions are equivalent:
\begin{itemize}
\item[(i)] The exact sequence 
\begin{eqnarray}
 \label{transseq}
 0 \ra \iota(A) \ra B \ra \id{T} \ra 0,
\end{eqnarray}
is split.
\item[(ii)] The jump-set $J(B) = \emptyset$,
\item[(iii)] The socle $S(B)$ is straight,
\item[(iv)] $\sexp(B) = \exp(B)$,
\end{itemize}
Moreover, if $(A,B)$ is a transition verifying the above conditions and
$\exp(A) = q$, then $\exp(B) = q$.
\end{lemma}
\begin{proof}
The conditions (ii) and (iv) are obviously equivalent: if $q = \sexp(B) =
\exp(B)$, then the exponent $q_i = \ord(T^i b) = q$ are all equal, and
conversely, if these exponents are equal, then $\sexp(B) = \exp(B)$: to see
this, consider $x \in B \setminus p B$ such that $\ord(x) = \sexp(B)$. Since
$0 \neq \overline{x} \in R(B)$, Lemma \ref{byun} shows that $x = T^k v(T), k
\geq 0, v \in (\Z_p[ T ])^{\times}$. Therefore $\ord(x) = q_k = q$, as
claimed.

Suppose that $S(B)$ is straight, so $\psi(b)$ generates $S(B)$. Then $T^j
\psi(b) \neq 0$ for all $0 \leq j < r'$ and thus $T^j \varpi(b) \cdot b =
\varpi(b) (T^j b) \neq 0$. Since $\ord(\varpi(b) T^j b) \leq p$, it follows
that the order is $p$ and $\varpi(b) = \varpi(T^j b)$, so $\ord(b) = \ord(T^j
b)$ for all $j$, and thus $\sexp(B) = \exp(B)$. Hence $(iii) \Rightarrow (ii),
(iv)$. Conversely, suppose that the socle is folded. Then let $\psi(T^k b)$ be
a generator of the socle, $k > 0$. The same argument as above shows that
$\ord(T^{k-1} b) > \ord(T^k b)$ and thus $J(B) \neq \emptyset$. Therefore
$(ii)-(iv)$ are equivalent.

We show that \rf{transseq} is split if $(ii)-(iv)$ hold. Suppose that
$\iota(a) \not \in p B$; then $\iota(a) = \nu b$ has non trivial image in
$R(B)$ and thus $\ord(a) = \ord(b)$. If \rf{transseq} is not split, then
$\psi(\iota(a)) \in \sum_{i=0}^{\deg(\nu)-1} \Z_p T^i b$, in contradiction to
$S(B)$ being straight. Thus $J(b) = \emptyset$ implies \rf{transseq} being
split.

Conversely, we show that if \rf{transseq} is split, then $S(B)$ is straight 
and $\sexp(B) = \exp(B)$. We have $B = \iota(A) \oplus B_{r'-r-1}$ and $S(B) 
= S(B_{r'-r-1}) \oplus \iota(S(A))$. On the other hand, $\iota(S(A)) = 
\iota(A) \cap K \neq \emptyset$; therefore $S(B_{r-1}) \cap S(K) = 
\emptyset$, and it follows that $S(B)$ is straight. This completes the proof
of the equivalence of $(i)-(iv)$.

If $\exp(A) = q$ and the above conditions hold, then $\exp(B) = \sexp(B)$ and
thus $\ell(B) = \prk(B) = r'$. We prove by induction that $r' = p d$: Assume
thus that $\prk(A) = d$ and let $s' = (q/p) b \in S(B)$, a generator. Then
$s:= N s' = (q/b) \iota(a) \in \iota(S(A))$ will be a generator of
$\iota(S(A))$. A rank comparison then yields 
\[ r' = \prk(S(B)) = \prk(S(A)) + (p-1) d = p d. \]
From $\iota(a) = p b u(\omega) + \omega^{p-1}
b$, we gather that $\ord(a) \geq \max(p \ord(b), \ord(\omega^{p-1} b))$. Since
$r' = \prk(B) = p d$ and $\ell(B) = r'$, it follows that $\ord(\omega^{p-1} b)
= \ord(b) = \ord(a) = q$. The claim follows by induction on the rank of $A$.
\end{proof}

We have seen in the previous lemma that regular flat transitions can be
iterated indefinitely: this is the situation for instance in $\Lambda$-modules
of unbounded rank: note that upon iteration, the exponent remains equal to the
exponent of the first module and this may be any power of $p$. The regular
wild transitions will be considered below, after the next lemma that
generalizes Lemma \ref{ab} to the case of increasing ranks, and gives
conditions for a large class of terminal transitions.
\begin{lemma}
\label{hasroot}
Suppose that $q':= \ord(a) > p$, $r' > r$ and $\iota(a) \in p B$. 
For $b \in B$ with $N b = a$, 
we let the module $C = C(b) := \sum_{i=0}^{r-1} \Z_p T^i b$. Then
\begin{itemize}
\item[ 1. ] 
\begin{eqnarray}
\label{prep1}
C \supset \iota(A) \quad \hbox{ and } \quad \iota(A) = p C.
\end{eqnarray}  
\item[ 2. ]
The element $b$ spans $B$ as a cyclic $\Z_p[ T ]$ -module and 
\begin{eqnarray}
\label{pn}
K = S(B). 
\end{eqnarray}
\end{itemize}
Moreover, $r' \leq (p-1) d$ and the transition $(A, B)$ is terminal.
\end{lemma}
\begin{proof}
  Let $a' = \iota(a)$ and $c \in B$ be maximal and such that $p^e c =
  a'$, thus $T^i p^e c = T^i a'$. Let $C = \sum_{i=0}^{r-1} \Z_p T^i
  c$. Since $T^i a' = p^e T^i c$, we have $\iota(A) \subset C$ and
  thus $\prk(C) \geq \prk(A)$; on the other hand, the generators of
  $C$ yield a base for $C/p C$, so the reverse inequality $\prk(C)
  \leq \prk(A)$ follows; the two ranks are thus equal.

  We show that $N : C \ra A$ is surjective. We may then apply the
  lemma \ref{ab} to the couple of modules $A, C$. Let $x \in \Z_p[ T
  ]$ be such that $N(c) = x a$. If $x \in (\Z_p[ T ])^{\times}$, then
  $N(x^{-1} c) = a$ and surjectivity follows.

Assume thus that $x \in \eu{M} = (p,T)$. We have an expansion
\[ N(c) = h(T) a = (h_0 + \sum_{i=1}^{r-1} h_i T^i) a , \quad h_i \in
\Z_p, \] and we assume, after eventually modifying $h_0$ by a $p$ -
adic unit, that $h_0 = p^k$ for some $k \in \N$. If $k = 0$, then
$h(T) \in (\Z_p[ T ])^{\times}$, so we are in the preceding case, so
$k > 0$. We rewrite the previous expansion as
\begin{eqnarray}
 \label{ep1}
N(c) = (p^k + T g(T)) a, 
\end{eqnarray}
with $g(T) \in \Z_p[ T ]$ of degree $< r-1$. Let $f = e+k-1$; from
$p^e c = a$, we deduce:
\[ p^f c = p^{k-1} \cdot (p^e c) = p^{k-1} a \quad \hbox{ and } \quad
N(p^f c) = N(p^{k-1} a) = p^k a .\] By dividing the last two
relations, we obtain $(1-p^f) N(c) = T g(T) a$. Since $B$ is finite,
we may choose $M > 0$ such that $p^{M f} c = 0$. By multiplying the
last expression with $(1-p^{Mf})/(1-p^f)$ we obtain
\[ N(c) = T g(T)(1 + p^f + \ldots ) a .\]
We compare this with \rf{ep1}, finding
$T g(T)(p^f + p^{2f } + \ldots ) a = p^k a$.

Since $\iota(a) \in p B$, we have $e > 0$. It follows that 
\[ p^k \cdot \left(1 - p^{e-1} T g(T)(1+p^f+ \ldots )\right) a = 0,\]
so $p^k a = 0$ - since the expression in the brackets is a
unit. Introducing this in \rf{ep1}, yields: $ N(c) = T g(T) a$.  From
$p^e c = a$, we then deduce $N(p^e c) = p a = p^e T g(T) a$ , and this
yields $p (1 - p^{e-1} T g(T) ) a = 0$. It follows from $e > 0$ that
$p a = 0$, in contradiction with the hypothesis that $\ord(a) > p$.
We showed thus that if $\iota(a) \in p B$ and $\ord(a) > p$, the norm
$N : C \ra A$ is surjective and we may apply the lemma \ref{ab}. Thus
$p C = A = N(C)$ and $p c = a$.

The module $B$ is $\Z_p[ T ]$-cyclic, so let $b$ be a generator with
$N b = a$ and let $\tilde{C} = \sum_{i=0}^{r'-1} T^i c$. We claim that
$\tilde{C} = B$.  For this we compare $R(B)$ to $R(\tilde{C})$; we
obviously have $R(\tilde{C}) \subseteq R(B)$. If we show that this is
an equality, the claim follows from Nakayama's Lemma \ref{naka}. The
module $R(B)$ if $\F_p[ T ]$ cyclic, so there is an integer $k \geq 0$
with $\overline{c} \in T^k R(B)$. But then $N (\tilde{C}) \subset T^k
N(B)$ and since $N : C \ra A$ is surjective and $C \subset \tilde{C}$,
we must have $k = 0$, which confirms the claim and completes the proof
of point 1.

Note that $\iota(S(A)) \subset A = p C$, thus $C \cap K \supseteq
\iota(S(A))$. Conversely, if $x \in C \cap K$, then $x \in p C$, since
$T^i c \not \in \iota(A)$ for $0 \leq i < r$, from the assumption $a
\in p B$. Therefore $x \in p C \cap K = \iota(A) \cap K =
\iota(S(A))$, as shown in \rf{sockern}, and
\[ C \cap K = \iota(S(A)) . \] If $r' > r$, then $\prk(K) = r'-r$; if
$r = r'$, the transition is stable and $K \subset C$.

We now prove \rf{pn}. Let $x = g b \in K, g \in a^{\top} \subset \Z_p[ T
]$. Since $p b \in A$, we have $ p x = g p b \in g A = 0$.  Thus $K \subset
S(B)$; conversely,
\begin{eqnarray*}
 \prk(S(B)) & = & \prk(S(A)) + r'-r = \prk(S(A)) + (\prk(B)-\prk(C)) \\ & = &
\prk(S(A)) + \prk(K)
\end{eqnarray*}
and since $S(A) \subset K$ and $S(B)$ is cyclic, it follows that $S(B)
= K$, which confirms \rf{pn} and assertion 2.

Finally, note that $\prk(S(B)) = r'$ and since $\kappa = \omega b$ generates
the socle and
\[ 0 = ((\omega+1)^p-1) b = \omega^{p-1} \kappa = T^{(p-1)d} \kappa, \]
it follows that $r' \leq (p-1) d$, as claimed.
\end{proof}
We now investigate regular wild transitions and show that not more than two
such consecutive transitions are possible.
\begin{lemma}
\label{wild}
Let $(X, A)$ be a wild transition with $\prk(X) = 1$ and $(A, B)$ be a
consecutive transition. Then
\begin{itemize}
\item[ 1. ] $S(A) \not \subset K(A)$ and there is an $x' \in X$ with $\ord(x')
= p^2$ together with $g = T f(T) a \in K(A)$ such that $s = \iota(x') + g$ is a
generator of $S(A)$ and $\ell(A) = 2$. 
\item[ 2. ] The rank $r' \leq (p-1) d$ and $B$ is terminal, allowing an
  annihilator $f_B(T) = T^{r'} - q/p w(T)$.
\end{itemize}
\end{lemma}
\begin{proof}
  In this lemma we consider two consecutive transitions, so we write
  $T = \tau - 1$, acting on $A, B$ and $\omega = (T+1)^p-1$
  annihilating $A$ and acting on $B$. We shall also need the norm
  \begin{eqnarray*}
 \id{N} & = & N_{B/X} = \frac{1}{T}((T+1)^{p^2} - 1) = p^{2} U(T) + p
  T^{p} V(T), \\
 U, V & = & 1 + O(T) \in (\Z_p[ T ])^{\times}   
  \end{eqnarray*}
  We
  let $q = \exp(A), q/p = \exp(X)$ and $q p = \exp(B)$.  In the wild
  transition $(A, B)$, the socle has length $p$ and if $s \in B$ is a
  generator, then $0 \neq N s = T^{p-1} s \in S(X)$; it follows that
  $s \not \in K(A)$. Let $\iota(x) = N s \in \iota(S(X))$; if $x \not
  \in p X$, there is a $c \in \Z_p^{\times}$ with $T^{p-1} s = c N(a)
  = c p u(T) a + c T^{p-1} a$ and thus $T^{p-1} (s - c a) u^{-1}(T) =
  p a $; then $p a \in K(A) \cap \iota(S(X))$, and thus $\ord(a) = p^2$ and Lemma \ref{hasroot} implies that $\prk(A) < p$, which is a
  contradiction to our choice. Therefore $\exp(X) > p$ and
  there is an $x' \in X$ of order $p^2$ such that $p \iota(x') =
  N(\iota(x')) = N(s)$, so we conclude that $N(s-x') = 0$ and $x'-s
  \in K(A)$, which implies the first part in claim 1. We show now that
  $\ell(A) = 2$; indeed, $s-q/p^3 \iota(x) \in K(A)$, so there is a power $p^k$ such that $s-q/p^3 \iota(x) = T p^k w(T) b$ and $w \in \Z_p[ T ] \setminus p \Z_p[ T ]$. From $T s = T^2 p^k w(T) b$, and since $\prk( T s \F_p[ T ]) =
  p-1$, we conclude that $w \in (\F_p[ T ])^{\times}$. Moreover the
  above identities in the socle imply
\[ q /p^{k} \geq p^2 \ord(T p^k a) > p = \ord(T^2 p^k a) \geq 
\ord(T^{p-1} p^k a) = q/p^{k+1}. \]
Consequently, $p^k = q/p^2$ and $\ord(T a) = q$ while $\ord(T^2 a) = q/p$, thus
$\ell(A) = 2$.

For claim 2 we apply point 1. in Lemma \ref{transmod}. Let $q =
\exp(A) = p \exp(X) > p^2$ and $\ell' = \ell(B) = \prk(q B)$, $\ell =
\ell(A) = \prk(q/p A)$. From the cyclicity of the socle, we have
\begin{eqnarray}
S(B)[ T ] & = & q/p^2 x \F_p = q/p^2 \id{N}(b) \F_p, \quad \hbox{hence
  $\exists c \in \F_p^{\times}$}, \nonumber \\
\label{prev}
T^{\ell'-1} q b & = & c q/p^2 \id{N}(b) = (c q U(T) + c q/p T^p V(T)) b.
\end{eqnarray}
Assuming that $\ell' > 1$, then $q b( 1 - O(T) ) = q/p T^p V(T) b$ and
thus $q/p T^p b = q b V_1(T)$. Then $\ord(T^p b) = q$ and thus
$\ell' < p$. Moreover $q/p T^{p + \ell'} b = q T^{\ell'} p V_1(T) b =
0$, thus $\ord(T^{p+\ell'} b) \leq q/p$ and a fortiori
\begin{eqnarray}
\label{tbound}
\ord(T^{2p-1} b) \leq q/p, \quad \ord(T^{p-1} b) \leq q. 
\end{eqnarray}

We now apply the norm of the transition $(A,B)$, which may be
expressed in $\omega$ as $N_{B/A} = p u(\omega) + \omega^{p-1}$. Note
that $u(\omega) = v(T) \in (\Z_p[ T ])^{\times}$, for some $v$
depending on $u$. Also
\begin{eqnarray*}
\omega & = & T^p + p T u(T) = T (T^{p-1} + p u(T)) \\ \omega^{p-1} & = &
T^{p(p-1)} + T^{(p-1)^2} p (p-1) u(T) + O(p^2).
\end{eqnarray*}

Since $p \geq 3$, \rf{tbound} implies $\ord(T^{p(p-1)} b) \leq q/p$ and
$\ord(T^{(p-1)^2} b) \leq q$. We thus obtain:
\begin{eqnarray*}
\iota(a) & = & p v(T) b + (T^{p(p-1)} + p u_1(T) T^{(p-1)^2} + O(p^2)) b, \quad
\hbox{hence}  \\  
q/p \iota(a) & = & q b v(T), 
\end{eqnarray*}
and it follows in this case that $1 < \ell = \ell' = 2 < p$.  Consider now
the module $Q = \iota(A)/(\iota(A) \cap p B)$. Since $p \iota(A)
\subset (\iota(A) \cap p B)$, this is an $\F_p[ T ]$ module; let $T^i$
be its minimal annihilator. Then $T^i \iota(a) v_1(T) = p b, v_1 \in
(\Z_p[ T ])^{\times}$; but $q/p v^{-1}(T) \iota(a) = q b$, and thus
$q/p (T^i - v_2(T)) \iota(a) = 0, v_2 \in (\Z_p[ T ])^{\times}$. If $i
> 0$, then $v_2(T) - T^i \in (\Z_p[ T ])^{\times}$ and this would
imply $q/p \iota(a) = 0$, in contradiction with the definition $q =
\ord(a)$.  Consequently $i = 0$ and $\iota(a) \in p B$. We may thus
apply the Lemma \ref{hasroot} to the transition $(A, B)$. It implies
that the transition is terminal and $r' < (p-1) d$.

Finally we have to consider the case when $\ell' = 1$, so the relation
\rf{prev} becomes
\[ q (1/c- U(T)) b = q/p T^p V(T)) b. \] If $c \neq 1$, then $q/p T^p
b = q w(T), w \in (\F_p[ T ])^{\times}$ and the proof continues like
in the case $\ell' > 1$. If $c = 1$, then we see from the development
of $U(T) = 1 + T \binom{p^2}{2} + O(T)^2$ that there is a unit $d =
\frac{p^2-1}{2} \in \Z_p^{\times}$ such that
\[ q d T U_1(T) b = q/p T^{p} V(T) b = 0, \] since $\ell = 1$ and thus
$T q b = 0$. We may deduce in this case also that $q/p \iota(a) = q b
\cdot c_1, \ c_1 \in \Z_p^{\times}$ and complete the proof like in the
previous cases. The annihilator polynomial of $B$ is easily deduced
from Lemma \ref{hasroot}: $T^{r'} b \in \iota(S(A))$ is a generator of
the last socle, so $T^{r'} b = T q/p^2 \iota(a) + c q/p^{3} \iota(x)$
and some algebraic transformations lead to
\[ T^{r'} b w(T) = q/p b, \quad w \in (\Z_p[ T ])^{\times}, \]
which is the desired shape of the minimal polynomial. Note that $q/p =
\exp(X)$; also, the polynomial is valid in the case when $(A,B)$ is stable. We
could not directly obtain a simple annihilating polynomial for $A$, but now it
arises by restriction.
\end{proof}
The previous lemma shows that an initial wild regular transition
cannot be followed by a second one. Thus growth is possible over
longer sequences of transitions only if all modules are regular
flat. The following lemma considers the possibility of a wild
transition following flat ones.
\begin{lemma}
\label{finflat}
Suppose that $(A,B)$ is a transition in which $A$ is a regular flat module of
rank $\prk(A) \geq p$ and $\exp(B) > \exp(A)$. Then $B$ is terminal and $d <
r' < (p-1) d$. Moreover, there is a binomial $f_B(T) = \omega T^{r'-d} - q
w(T) \in b^{\top}; w \in (\F_p[ T ])^{\times}$.
\end{lemma}
\begin{proof}
  Since $\exp(B) > \exp(A)$, the transition is not flat. Assuming that
  $B$ is not terminal, then it is regular wild. Let $q = \exp(A) =
  \sexp(A)$ and $s' \in S(B)$ be a generator of the socle of $B$. By
  comparing ranks, we have $T^{(p-1)d} s' = \omega^{p-1} s' = q/p
  \iota(a) v(T), v \in (\F_p[ T ])^{\times}$. If $q = p$, then \[
  T^{(p-1)d} s' = \nu v_1(T) b \quad \Rightarrow \quad \nu s' = p
  u(\omega) s' + \omega^{p-1} s' = \nu v_1(T) b, \] and thus $s' -
  v_1(T) b \in K$. Since $K = \overline{\omega} B$, we have
  $(\overline{\omega} x + v_1(T)) b \in S$. The factor $
  \overline{\omega} x + v_1(T) \in (\Z_p[ T ])^{\times}$ and it
  follows that $b \in S(B)$, which contradicts the assumption $\exp(B)
  > \exp(A)$, thus confirming the claim in this case. If $q > p$, then
  the previous identity yields $s' - q/p^2 \iota(a) \in K$ and thus
  $s' = q/p^2 \iota(a) + \omega x b$; we let $x=p^k v(T)$ with $v(0)
  \not \equiv 0 \bmod p$.

We assumed that $\prk(S(B)) = p d$, so it thus follows that $v(T) \in (\Z_p[ T
])^{\times}$. Recall that $\ord(b) = q p$ as a consequence of $\exp(B) >
\exp(A)$ and \rf{exps}; the norm shows that
\[ q u(\omega) b + (q/p) \omega^{p-1} b = (q/p) \iota(a) \neq 0.\]
 If $(q/p) \omega^{p-1} b = 0$, then $q u(\omega) = (q/p) \iota(a)$ which implies
 that the annihilator of $Q = \iota(A) /(p B \cap \iota(A))$ is trivial and
 $\iota(A) \subset p B$. We are in the premises of Lemma \ref{hasroot}, which
 implies that $B$ is terminal. 

It remains that $\ord(\omega^{p-1} b) = q$. We introduce this in the
expression for the generator of the socle:
\[ s' = \omega p^k v(T) b - q/p u(\omega) b - q/p^2 \omega^{p-1} b \in
S(B). \] We have
\[ q/p^{k-1} = \ord(p^k b) \geq p^2 = \ord(p^k \omega b) \geq 
\ord( p^k \omega^{p-1} b) = q/p^k, \]
and thus $q/p^2 \leq p^k \leq q/p$. Note that 
\[ \omega s' = \omega^2 p^k v(T) b + \omega q/p^2 \iota(a) = \omega^2
p^k v(T) b \in S(B); \] from $\omega^{p-1} (q/p) b \in S(B)$, it
follows that $p^k = q/p$ and $\ord(\omega b) = q p$ while
$\ord(\omega^2 b) = q = \ord(\omega^{p-1} b)$. Let $i > 0$ be the
least integer with $q/p \omega T^i b \in S(B)$. From the definition of
$s'$ we see that $T^k$ is also the annihilator of $(q/p^2) \iota(a)$ in
$\iota(A)/(S(\iota(A)))$, so $i = d$, since $A$ is flat. It follows
that $\ell(B) = \prk(B/q B) = 2 d$ and the cyclicity of the socle
implies that 
\[ T^{(p-2) d} s' =  q b v_1(T) = T^{(p-2) d} s' = 
\frac{q}{p} \omega^{p-1} v(T) b + \frac{q}{p^2} \omega^{p-2} \iota(a). \] 
Hence there is a unit $v_2(T)
\in (\Z_p[ T ])^{\times}$ such that $\frac{q}{p} ( p - \omega^{p-1}
v_2(T)) b = 0$. By comparing this with the norm identity $\frac{q}{p}
( p + \omega^{p-1} u^{-1}(T)) b = \frac{q}{p} \iota(a)$ we obtain,
after elimination of $\omega^{p-1} b$, that $\frac{q}{p}(p b v_3(T) +
\iota(a)) = 0$ and the reasonment used in the previous case implies
that $\iota(a) \in p B$ so Lemma \ref{hasroot} implies that $B$ is
terminal.

We now show that $\iota(A) \subset p B$. Otherwise, $r' \geq (p-1) d$
and $T^{r'-1} s' = c T^{d-1} q/p \iota(a)$, so by cyclicity of the
socle, $T^{r'-d} s' = v(T) q/p \iota(a)$ while $\nu s' = T^{p d - r' -
  1} \iota(a) \neq 0$. A similar estimation like before yields also in
this case $s' = q/p \omega v(T) b + q/p^2 T^j \iota(a), j = p d - r'
-1$. Then $\omega^{p-2} s' = q/p \omega^{p-1} v(T) b \in S(B)$. Let $i
> 0$ be the smallest integer with $p b \in \iota(a)$. Then $q b = T^i
q/p \iota(a) v_1(T)$ and we find a unit $v_2(T)$ such that
\[ q/p (\omega^{p-1} T^{i+r'-(p-1)d-1} - p v_2(T)) b = 0. \] This
implies by a similar argument as above, that $\iota(a) \in p
B$. Therefore we must have $r' < (p-1)d$ and $S(B) = K(B) \supset q/p
\iota(A)$. Since $q b = q/p \iota(a) = \omega b T^{r'-d} v(T)$, we
obtain an annihilator polynomial $f_B(T) = T^{r'-d} - q v^{-1}(T)$,
which completes the proof of the lemma.
\end{proof}
We finally apply the Lemma \ref{lreg} to a sequence of flat
transitions. This is the only case which allows arbitrarily large growth of
the rank, while the value of the exponent is fixed to $q$.
\begin{lemma}
\label{lregind}
Suppose that $A_1, A_2, \ldots, A_n$ are a sequence of cyclic $\Z_p[ T ]$
modules such that $(A_i, A_{i+1})$ are conic non-stable transitions with
respect to some $\omega_i \in \Z_p[ T ]$ and $\prk(A_1) = 1$. If $n > 3$, then
$A_i$ are regular flat for $1 \leq i < n$.
\end{lemma}
\begin{proof}
If $n = 2$ there is only one, initial transition: this case will be considered
in detail below. Assuming that $n > 2$, the transitions $(A_i, A_{i+1})$ are
not stable; if $(A_1, A_2)$ is wild, then Lemma \ref{wild} implies $n \leq
3$. The regular transitions being by definition the only ones which are not
terminal, it follows that $A_k, k = 1, 2, \ldots, n-1$ are flat. Lemma
\ref{finflat} shows in fact that $A_2$, which must be flat, can only be
followed by either a regular flat or a terminal transition. The claim follows
by induction.
\end{proof}

\subsection{Case distinctions for the rank growth}
We have gathered above a series of important building blocks for analyzing
transitions. First we have shown in point 2. of Lemma \ref{transmod} that all
transitions that are not flat are terminal. Thus for the cases of interest
that allow successive growths of ranks, we must have $r = d, r' = p d$. The
Lemma \ref{lreg} shows that these reduce to $\exp(A) = q = \sexp(A)$.

We start with an auxiliary result which will be applied in both remaining
cases:
\begin{lemma}
\label{sk}
Let $(A, B)$ be a transition with $\exp(A) = p$ and $r = d < r' < p d$. Then
$S(B) \supseteq K$ with equality for $r' \leq (p-1) d$. If $S(B) \neq K$, then
$s = T^{pd - r'}$ is a generator of $S(B)$.
\end{lemma}
\begin{proof}
Since $r = d = \deg(\omega)$, it follows that $N(T^i b)) = T^i \iota(a) \neq 
0$ for $i = 0, 1, \ldots, r-1$ while $N(T^d b) = T^d \iota(a) = 0$, since 
$\exp(A) = p$. Therefore $K = \omega B$ and $R(K) = T^d R(B)$. 

We first show that $S(B) \subset K$: let $s \in S(B)$ be a generator.
Cyclicity of the socle implies that $T^{r'-1} s = c_0 T^{d-1} \iota(a) \neq 0,
c_0 \in \F_p^{\times}$ and $T^{r'} s = 0$. We have $N(s) = (p u(\omega) +
\omega^{p-1}) s = \omega^{p-1} s = T^{(p-1) d} s$. If $r' \leq (p-1) d$, then
$T^{(p-1) d} s = 0$ and thus $s \in K$ and $S(B) = S(K)$. Otherwise, $T^{(p-1)
d} \overline{b} = \overline{\iota(a)} \in R(B)$ and a fortiori $N(s) \in
\iota(S(A)) = \F_p[ T ] \iota(a)$. Let $0 \leq k < d$ be such that $N(s) = T^k
\iota(a)$: we may discard an implicit unit by accordingly modifying $s$. Then
$T^{d-(k+1)} N(s) \neq 0$ and $T^{d-k} N(s) = 0$.  Therefore $T^{d-(k+1)} s
\not \in K = \omega B$ and $s \not \in p B$. It follows that $\overline{s}
R(B) = T^k R(B)$ and $s = T^k v(T) b$, by lemma \ref{byun}. By comparing
ranks, we see that $k = p d - r'$ and $s = T^{pd-r'}$ is a generator of
$S(B)$. 
\end{proof}

For initial transitions we have:
\begin{lemma}
\label{initrans}
Let $(A,B)$ be a conic transition and suppose that $r = 1$ and the transition
is terminal. If $r' < p$, then $B$ has a monic annihilator polynomial $f_B(T)
= T^{r'}-q w(T)$ with $q = \ord(a)$ and $w \in (\Z_p[ T ])^{\times}$.
\end{lemma}
\begin{proof}
We let $q = \ord(a)$ throughout this proof. Assume first that $r' < p-1$, so
$T^{p-1} b = T^{p-1-r'} ( T^{r'} b ) = 0$, since $T^{r'} a \in \iota(A)$ by
definition of the rank. Then $\iota(a) = (p u(T) + T^{p-1}) b = p u(T) b$ and
$p b = u(T)^{-1} \iota(a) = \iota(a)$, since $u(T) \equiv 1 \bmod T$. We have
thus shown that $\iota(a) \in p B$ and, for $q > p$, we may apply Lemma
\ref{hasroot}. It implies that $S(B) = K = T B$ and $T^{r'-1} (T b) = c q/p
\iota(a) = c q b, c \in \Z_p^{\times}$. This yields the desired result for
this case. If $q = p$, the previous computation shows that $\iota(a) \in p B$,
but Lemma \ref{hasroot} does not apply here. We can apply Lemma \ref{sk}, and
since, in the notation of the lemma, $d = 1$ and the transition is assumed not
to be stable, we are in the case $1 < r' \leq p-d$ and thus $S(B) = K$
too. The existence of the minimal polynomial $f_B(T) = T^{r'} - p c \in
b^{\top}$ follows from this point like in the case previously discussed.

If $r' = p-1$, then $T^{p-1} b = \iota(a) - p u(T) b $ and thus $T^p b = p T
u(T) b = 0$ and $T b \in S(B)$. Since the socle is cyclic and $K = T B$ it
follows that $K = S(B)$. In particular, there is a $c \in \Z_p^{\times}$ such
that
\begin{eqnarray*}
 T^{p-1} b & = & \frac{c q}{p} \iota(a) = \frac{c q}{p} (p u(T) + T^{p-1}) b \\ 
 & = & c q u(T) b + \frac{c q}{p} u(T) T^{p-1} b, \quad \hbox{hence}
\\ 
T^{p-1}(1- \frac{c q}{p} u(T) ) b & = & c q u(T) b. 
\end{eqnarray*}
If $q > p$, then $1-\frac{c q}{p} u(T) \in (\Z_p[ T ])^{\times}$; if $q = p$,
it must also be a unit: otherwise $1-c u(T) \equiv 0 \bmod T$ and thus $T^p b
= c q u(T) b = 0$, in contradiction with the fact that $q b = q/p \iota(a)
\neq 0$. In both cases we thus obtain an annihilator polynomial of the shape
claimed.

Finally, in the case $r' = p$ and the transition is wild. We refer to Lemma
\ref{wild} in which treats this case in detail.
\end{proof}

\begin{remark}
\label{rthaine}
Conic $\Lambda$-modules are particularly simple modules. The following example
is constructed using Thaine's method used in the proof of his celebrated
theorem \cite{Th}. Let $\F_1 \subset \Q[ \zeta_{73} ]$ be the subfield of
degree $3$ over $\Q$ and $\K_1 = \F_1 \cdot \Q[ \sqrt{-23} ]$. Then $A_1
:=(\id{C}(\K_1))_3 = C_9$ is a cyclic group with $9$ elements. If $\K_2$ is
the next level in the cyclotomic $\Z_3$-extension of $\K_1$, then
\[ A_2:= (\id{C}(K_2))_3 = C_{27} \times C_9 \times C_9 \times C_3 \times C_3
\times \times C_3 . \] The prime $p = 3$ is totally split in $\K_1$ and the
classes of its factors have orders coprime to $p$. Although $A_1$ is $\Z_p$-
cyclic, already $A_2$ has $p$-rank $2 p$. Thus $\rg{A}$ cannot be conic, and
it is not even a cyclic $\Lambda$-module. 

It is worth investigating, whether the result of this paper can extent to the
case when socles are not cyclic and conicity is not satisfied, in one or more
of its conditions. Can these tools serve to the understanding of
$\Lambda$-modules as the one above?
\end{remark}

\section{Transitions and the critical section}
We return here to the context of $\Lambda$ modules and conic elements,
and use the notation defined in the introduction, so $\id{A}_n =
\Lambda a_n$ are the intermediate levels of the conic $\Lambda$-module
$\Lambda a \subset \rg{A}^-$. We apply the results of the previous
chapter to the transitions $C_n = (\id{A}_n, \id{A}_{n+1})$ for $n <
n_0$. By a slight abuse of notation, we keep the additive notation for
the ideal class groups that occur in these concrete transitions. The
first result proves the consistency of the models:
\begin{lemma}
\label{cons}
Let the notations be like in the introduction and $a =(a_n)_{n \in \N}
\in \rg{A}^-$ a conic element, $\id{A} = \Lambda a$ and $\id{A}_n =
\Lambda a_n \subset A_n^-$. Then the transitions $(\id{A}_n,
\id{A}_{n+1})$ are conic in the sense of Definition \ref{dcontrans},
for all $n > 0$.
\end{lemma}
\begin{proof}
  Let $A = \id{A}_n, B = \id{A}_{n+1}$ and $N = \Norm_{\K_{n+1},\K_n},
  \iota = \iota_{n,n+1}$ be the norms of fields and the ideal lift
  map, which is injective since $a \in \rg{A}^-$. We let $T = \tau-1$
  with $\tau$ the restriction of the topological generator of $\Gamma$
  to $\K_{n+1}$ and $\omega = \omega_n = (T+1)^{p^{n-1}} - 1$. Then a
  fortiori $\omega A = 0$, and all the properties 1. - 3. of conic
  transitions follow easily. Point 5. is a notation. We show that the
  important additional property 4. follows from the conicity of
  $a$. The direction $\omega A \subset K$ follows from $Y_1 = TX $ in Theorem
  \ref{iw6}. The inverse inclusion is a consequence of point 1. of the
  definition of conic elements. Conversely, if $x \in K$, we may
  regard $x = x_{n+1} \in \id{A}_{n+1}$ as projection of a norm
  coherent sequence $y = (x_m)_{m \in \N} \in \id{A}$: for this we
  explicitly use point 3 of the definition of conic elements. Since $x
  = y_{n+1} = 1$ we have by point 2 of the same definition, $y \in
  \omega_n \cdot \id{A}$. This implies $y_n = N x = 1$; this is the
  required property 4 of Definition \ref{dcontrans}
\end{proof}

The next lemma relates $v_p(a_1)$ to the minimal polynomials $f_a(T)$:
\begin{lemma}
\label{val0}
Let $a \in \rg{A}^-$ be conic and $m = v_p(a_1)$. Then $v_p(f_a(0)) =
m$.  In particular, if $v_p(a_1) = 1$, then $f_a(T)$ is an Eisenstein
polynomial.
\end{lemma}
\begin{proof}
  Let $q = p^m$ and $b = q a \in \Lambda a$. Then $b_1 = 0$ and, by
  conicity, it follows that $q a = b = T g(T) a$. It follows that $T
  g(T) - q$ annihilates $a$. We may choose $g$ such that $\deg(g(T) T)
  = \deg(f_a(T))$, so there is a constant $c \in \Z_p$ such that $T
  g(T) - q = c f_a(T)$. Indeed, if $c$ is the leading coefficient of
  $T g(T)$, the polynomial $D(T) = T g(T) - q - c f_a(T)$ annihilates
  $a$ and has degree less than $\deg(f_a)$. Since $f_a$ is minimal,
  either $D(T) = 0$, in which case $c = 1$ and $f_a(T) = T g(T)-q$,
  which confirms the claim, or $D(T) \in p \Z_p[ T ]$ and $c \equiv 1
  \bmod p$. Since $c$ is a unit in this case, we may replace $b$ by
  $c^{-1} b = T g_1(T) a$ and the polynomial $T g(T)$ is now
  monic. The previous argument implies that $f_a(T) = T g_1(T) -
  c^{-1} q$, which completes the proof. Since $f_a(T)$ is
  distinguished, we have $f_a(T) \equiv T^d \bmod p$ and if $m = 1$,
  then $p^2 \nmid f_a(0)$, so $f_a(T)$ is Eisenstein. The converse is
  also true.
\end{proof}

If $m > 1$, we have seen in the previous chapter that there are
minimal polynomials of $\id{A}_{n_0}$ which are essentially binomials;
in particular, they are square free. It would be interesting to derive
from this fact a similar conclusion about $f_a(T)$. We found no
counterexamples in the tables in \cite{EM}; however the coefficients
of $f_{n_0}(T)$ are perturbed in the stable growth too, and there is
no direct consequence that we may derive in the present setting. The
next lemma describes the perturbation of minimal polynomials in stable
growth:
\begin{lemma}
\label{lminpol}
Let $g_n(T) = T^r - p \tilde{g}_n(T)$ be a minimal polynomial of $\id{A}_n$. 
If $n \geq n_0$, then
\begin{eqnarray}
\label{etshift} 
g_n(T) \equiv f_a(T) \bmod \id{A}_{n-1}^{\top}
\end{eqnarray}
\end{lemma}
\begin{proof}
  The exact annihilator $f_a(T)$ of $\id{A}$ also annihilates all
  finite level modules $\id{A}_n$. In particular, for $n \geq n_0$ we
  have $\deg(g_n(T)) = \deg(f_a)$ for all minimal polynomials $g_n$ of
  $\id{A}_n$, and thus $\deg(g_n-f_a) < \lambda(a)$. We note that $g_n
  - f_a = p \delta_n(T) \in p \Z_p[ T ]$ with $\deg(\delta_n) < r$. It
  follows that
\[ 0 = p \delta_n(T) a_n = \delta_n(T) \iota_{n-1,n}(a_{n-1}), 
\]
 and since $\iota$ is injective, it follows that $\delta_n(T) \in
 \id{A}_{n-1}^{\top}$, as claimed.
\end{proof}

It is worthwhile noting that if $a$ is conic and $f_a(T) =
\prod_{i=1}^k f_i^{e_i}(T)$ with distinct prime polynomials $f_i(T)$,
then $b_i := f_a(T)/f_i(T) a$ have also conic transitions, but the
modules $\Lambda b_i$ are of course not complementable as
$\Lambda$-modules.
\subsection{Proof of Theorem \ref{main}}
With this, we shall apply the results on conic transitions and prove
the Theorem \ref{main}
\begin{proof}
  Let $a \in \rg{A}^- \setminus p \rg{A}^-$ be conic, let $n_0$ be its
  stabilization index and $\id{A}_n = \Lambda a_n, n \geq 0$ be the
  intermediate levels of $\id{A} = \Lambda a$. If $n_0 = 1$, then the
  Lemma \ref{fukuda} implies that $\lambda(a) = 1$ and $f_a(T)$ is
  linear. If $v_p(a_1) = 1$, then Lemma \ref{val0} implies that
  $f_a(T)$ is Eisenstein. Otherwise, the Lemmata
  \ref{wild} and \ref{initrans} imply that $\prk(\id{A}) < p(p-1)$
  with the exception of flat transition modules with high rank. The
  minimal polynomials at stabilization level are binomials, and this
  completes the proof.
\end{proof}

\subsection{Some examples}
We shall discuss here briefly some examples\footnote{I am grateful to
 an anonymous referee for having pointed out some very useful
  examples related to the present topic.} drawn from the paper of
Ernvall and Mets\"ankyl\"a \cite{EM} and the tables in its
supplement. The authors consider the primes $p = 3$, $\rho = \zeta_p$
and base fields $\K = \K( m ) = \Q[ \sqrt{m}, \rho ]$. They have
calculated the annihilator polynomials of $f_a(T)$ for a large choice
of cyclic $A(\K(m))^-$. Here are some examples:
\begin{example}
\label{e1}
In the case $m = 2732$, $A^-(\K_1(m)) \cong C_{p^2}$ and
  $A^-(\K_2(m)) \cong C_{p^3} \times C_p$. The growth stabilizes and
  the polynomial $f_a(T)$ has degree $2$; the annihilator $f_2(T)$ is
  a binomial, but not $f_a(T)$, so the binomial shape is in general
  obstructed by the term $f_a(T) = f_2(T) + O(p)$.
\end{example}
\begin{example}
\label{e2}
In the case $m = 3512$, we have $A^-(\K_1(m)) \cong C_{p^2}$ and
$A^-(\K_2(m)) \cong C_{p^3} \times C_p \times C_p$. The polynomial $f_a(T)$
has degree $3$ and for $B = A(\K_2(m))^-$ and $A = A(\K_1(m))^-$. This is a wild transition, which is initial and terminal simultaneously. We did not derive a precise structure for such transitions in Lemma \ref{wild}. 
\end{example}
\begin{example}
\label{e3}
In the case $m = -1541$, the authors have found $\lambda =
4$. Unfortunately, the group $A(\K_3(m))$ cannot be computed with PARI, so our
verification restricts to the structure of the transition $(A, B) = (\id{A}_1,
\id{A}_2)$. This is the most interesting case found in the tables of \cite{EM}
and the only one displaying a wild initial transition. The Lemma \ref{wild}
readily implies that the transition $(\id{A}_2, \id{A}_3)$ must be terminal
and $\lambda < (p-1)p = 6$, which is in accordance with the data. The
structure is $\id{A}_2 = C_{p^3} \times C_{p^3} \times C_p$ and with respect
to this group decomposition we have the following decomposition of individual
elements in $A = \id{A}_1$ and $B = \id{A}_2$:
\begin{eqnarray*}
\begin{array}{ c c c c c c c c c c c}
b & = & (1, 0, 0) & & T b & = & (0, 10, 1) & & T^2 b & = & (-6, 9, 1) \\ 
T^3 b & = & (18, -3, 0) & & T^4 b & = & (18, 9, 0) & & T^5 b & = & (0, 9, 0) \\ 
a & = & (0, 12, 1) & & 3 a & = & (0, 9, 0) & & 9 b & = & (9, 0, 0). 
\end{array}
\end{eqnarray*}
Some of the particularities of this examples are: $S(B)$ is generated by $s' =
T^3 b - 2 a$ and it is $\F_p[ T ]$-cyclic, as predicted. Moreover, $s' \in K +
\iota(A)$ but $s \not \in K$ and $p b \not \in \iota(a)$, while $q b \in S(B)[
T ] = S(A)[ T ]$, both facts that were proved in the Lemma \ref{wild}.
\end{example}
\begin{example}
\label{e4}
In all further examples with $\lambda \geq 3$, the fields $\K(m)$ have more
than one prime above $p$ and $\rg{A}^-(m)$ is not conic. For instance, for $m
= 2516$, we also have $A^-(\K_1(m)) \cong C_{p^2}$ and $A^-(\K_2(m)) \cong
C_{p^3} \times C_p \times C_p$, but $T^3 b = 0$, for $b$ a generator of
$A(\K_2(m))^-$. The module is thus obviously not conic. This examples indicate
a phenomenon that was verified in more cases, such as our example in Remark
\ref{rthaine}: an obstruction to conicity arises from the presence of
\textit{floating} elements $b \in \rg{A}^-$. These are defined as sequences $b
= (b_n)_{n \in \N} \in \rg{A}^- \setminus (p, T) \rg{A}^-$ having $b_1 =
0$. When such elements are intertwined in the structure of $\Lambda a$, one
encounters floating elements. It is an interesting question to verify if the
converse also holds: $a \in \rg{A}^- \setminus (p, T) \rg{A}^-$ is conic if it
contains not floating elements. Certainly, the analysis of transitions in
presence of floating elements is obstructed by the fact that the implication
$T x = 0 \Rightarrow x \in A$ is in general false. However, the obstruction
set is well defined by the submodule of floating elements, which indicates a
possible extension of the concepts developed in this paper. The analysis of
floating elements is beyond the scope of this paper and will be undertaken in
subsequent research.
\end{example}
\begin{example}
\label{e5}
 Let $\K = \Q[ \sqrt{-31} ]$ with $A(\K) = C_3$ and only one prime
above $p = 3$. A PARI computation shows that $A(\K_2) = C_{p^2}$, so Fukuda's
Theorem implies that $\rg{A}$ is $\Lambda$-cyclic with linear annihilator. Let
$\KL/\K$ be the cyclic unramified extension of degree $p$. There are three
primes above $p$ in $\KL$ and $A(\KL) = \{ 1 \}$, a fact which can be easily
proved and needs no verification. Let $\KL_n = \KL \cdot \K_n$ be the
cyclotomic $\Z_p$-extension of $\KL$. One can also prove that $A(\KL_n) \cong
(A_n(\K))^p$, so $\rg{A}(\KL)$ is also $\Lambda$ cyclic with the same linear
annihilator polynomial as $\rg{A}(\K)$. Let $b \in \rg{A}(\KL)$ be a generator
of the $\Lambda$-module. The above shows that $b$ is a floating class.

The extension $\KL/\Q$ in this example is galois but not CM and $p$ splits in
$\KL/\K$ in three principal primes. If $\nu \in \Gal(\KL/\K)$ is a generator,
it lifts in $\Gal(\KH_{\infty}/\K)$ to an automorphism $\tilde{\nu}$ that acts
non trivially on $\Gal(\KH_{\infty}/\KL_{\infty})$.

Let $\B_{\infty}$ be the $\Z_p$-extension of $\Q$ and $\KH_{\infty}$
be the maximal $p$-abelian unramified extension of $\K_{\infty}$ and
of $\KL_{\infty}$ (the two coincide in this case); then the sequence
\begin{eqnarray}
\label{bsplit}
 0 \ra \Gal(\KL_{\infty}/\B_{\infty}) \ra
\Gal(\KH_{\infty}/\B_{\infty}) \ra \Gal(\KH_{\infty}/\KL_{\infty}) \ra
0 
\end{eqnarray}
is not split in the above example, and this explains why
$\tilde{\nu}$ lifts to a generator of $X' :=
\Gal(\KH_{\infty}/\K_{\infty})$.

Let $\eu{p}, \nu \eu{p}, \nu^2(\eu{p}) \subset \KL$ be the primes
above $p$ and $I_0, I_1, I_2 \subset \Gal(\KH_{\infty}/\KL)$ be their
inertia groups: then $I_1 = I_0^{\tilde{\nu}}, I_2 =
I_0^{\tilde{\nu}^2}$. Let $I \subset \Gal(\KH_{\infty}/\K)$ be the
inertia of the unique prime above $p$ and $\tau \in
\Gal(\KH_{\infty}/\K)$ be a generator of this inertia. We fix $\tau'$
as a lift of the topological generator of $\Gamma$: it acts in
particular also on $\KL$. Let $\tau$ be a generator of $I_0$ and $a
\in X = \Gal(\KH_{\infty}/\KL_{\infty})$ such that $\tau_1 = a \tau$
is a generator of $I_1$. We assume that both $\tau, \tau_1$ restrict
to a fixed topological generator of $\Gamma =
\Gal(\KL_{\infty}/\K_{\infty})$. Then
\[ \tau_1 = a \tau = \tau^{\tilde{\nu}} = \tilde{\nu}^{-1} \tau \tilde{\nu} \quad 
 \Rightarrow \quad a = \tilde{\nu}^{-1} \tau \tilde{\nu} \tau^{-1} .
\]
Since $\tau$ acts by restriction as a generator of $\Gamma' =
\Gal(\K_{\infty}/\K)$ and $\tilde{\nu}$ generates $X'$, the above computation
implies that $a \in (\Gal(\KH_{\infty}/\K))' = T X' = p X' = X$. In
particular, $a$ is a generator of $X \cong \rg{A}(\KL)$.

In this case we have seen that the primes above $p$ are principal, the module
$\rg{A}(\KL)$ is floating and it is generated by $a = \tau_1 \tau^{-1} \not
\in T X$. Thus $Y_1 = \Lambda a = \Z_p a$ and $[Y_1 : T X] = p$. Since $T X$
is the commutator, there must be a cyclic extension $\KL'/\KL$ of degree $p$
which is $p$-ramified but becomes unramified at infinity. It arises as
follows: let $\KH_2$ be the Hilbert class field of $\K_2$. Then $\KH_2/\KL_2$
is cyclic of degree $p$ and $\Gal(\KH_2/\K) = < \varphi(a_2) >$, with $a_2 \in
A(\K_2)$ a generator. Thus $(T - cp) a_2 = 0$ for some $c \in \Z_p^{\times}$
and $\Gal(\KH_2/\KL_2) = p < \varphi(a_2) > = < \varphi(T a_2) >$. Since $T^2
a_2 = c^2 p^2 a_2$, it follows that $T \Gal(\KH_2/\KL_2) = 0$ and thus
$\KH_2/\KL_1$ is abelian. This induces a cyclic extension $\KL'_1/\KL_1$ which
is $p$-ramified, but becomes unramified already over $\KL_2$. 

It also explains the role of the sequence \rf{bsplit} in Theorem \ref{simram}. Phenomena in this context will be investigated together with the question about floating classes in a subsequent paper.
\end{example}
The prime $p = 3$ is interesting since it immediately display the more
delicate cases $r' = p-1$ and $r = p$ in Lemma \ref{initrans}. We found no
examples with $\lambda > p$, which require an intermediate flat transition
according to the above facts.

\section{The ramification module}
In this section we prove the theorems stated in \S 1.2. The terms and
notations are those introduced in that introductory section. Note that the
choice of $\K$ as a galois CM extension containing the \nth{p} roots of unity
is useful for the simplicity of proofs. If $\K$ is an arbitrary totally real
or CM extension, one can always take its normal closure and adjoin the roots
of unity: in the process, no infinite modules can vanish, so facts which are
true in our setting are also true for subextensions of $\K$ verifying our
assumptions.

Let us first introduce some notations: $\KH_1$ is the $p$-part of the Hilbert
class field of $\K$ and $\overline{\KH}_1 = \KH_1 \cdot \K_{\infty}$;
$\Omega/\K$ is the maximal $p$-abelian $p$-ramified extension of $\K$. It
contains in particular $\K_{\infty}$ and $\zprk(\Omega/\KH_1) =
r_2+1+\id{D}(\K)$, where $\id{D}(\K)$ is the Leopoldt defect. Since $\K$ is
CM, complex multiplication acts naturally on $\Gal(\Omega/\K_{\infty})$ and
induces a decomposition 
\[ \Gal(\Omega/\K_{\infty}) = \Gal(\Omega/\K_{\infty})^+ \oplus
\Gal(\Omega/\K_{\infty})^-; \]
this allows us to define 
\begin{eqnarray}
\label{omdefs}
\Omega^- & = & \Omega^{\Gals(\Omega/\K_{\infty})^+} \\ 
\Omega^+ & = & \Omega^{\Gals(\Omega/\K_{\infty})^-} \nonumber,
\end{eqnarray} 
two extensions of $\K_{\infty}$. 

We shall review Kummer radicals below and derive a strong property of galois groups which are $\Lambda$-modules with annihilator a power of some polynomial: the order reversal property. Combined with an investigation of the galois group of $\Omega^-/\KH_1$ by means of class field theory, this leads to the proof of Theorem \ref{simram}.
\subsection{Kummer theory, radicals and the order reversal}
Let $\rg{K}$ be a galois extension of $\Q$ which contains the \nth{p} roots of
unity and $\KL/\rg{K}$ be a finite Kummer extension of exponent $q = p^m, m
\leq n$. Its classical Kummer radical $\rad(\KL/\rg{K}) \subset
\rg{K}^{\times}$ is a multiplicative group containing $(\rg{K}^{\times})^{q}$
such that $\KL = \rg{K}[ \rad(\KL)^{1/q} ]$ (e.g. \cite{La1}, Chapter VIII, \S
8). Following Albu \cite{Al}, we define the \textit{cogalois} radical
\begin{eqnarray}
\label{cog}
\Rad(\KL/\rg{K}) = \left([ \rad(\KL/\rg{K})^{1/q} ]_{\tiny{\rg{K}}^{\times}}\right) /
\rg{K}^{\times},
\end{eqnarray}
where $[ \rad(\KL/\rg{K})^{1/q} ]_{\tiny{\rg{K}}^{\times}}$ is the multiplicative
$\rg{K}^{\times}$-module spanned by the roots in $\rad(\KL/\rg{K})^{1/q}$ and
the quotient is one of multiplicative groups. Then $\Rad(\KL/\rg{K})$ has the
useful property of being a finite multiplicative group isomorphic to
$\Gal(\KL/\rg{K})$. For $\rho \in \Rad(\KL/\rg{K})$ we have $\rho^q \in
\rad(\KL/\rg{K})$; therefore, the Kummer pairing is naturally defined on
$\Gal(\KL/\tiny{\rg{K}}) \times \Rad(\KL/\tiny{\rg{K}})$ by
\[ < \sigma, \rho >_{\Rad(\KL/\tiny{\rg{K}})} = < \sigma, \rho^q >_{\rad(\KL/\tiny{\rg{K}})}. \]
Kummer duality induces a twisted isomorphism of $\Gal(\rg{K}/\Q)$ - modules
$\Rad(\KL/\rg{K})^{\bullet} \cong \Gal(\KL/\rg{K})$. Here $g \in
\Gal(\rg{K}/\Q)$ acts via conjugation on $\Gal(\KL/\rg{K})$ and via $g^{*} :=
\chi(g) g^{-1}$ on the twisted module $\Rad(\KL/\rg{K})^{\bullet}$; we denote
this twist the \textit{Leopoldt involution}. It reduces on $\Gal(\rg{K}/\K)$
to the classical Iwasawa involution (e.g. \cite{La}, p. 150).

We now apply the definition of cogalois radicals in the setting of Hilbert
class fields. Let $\K$ be like before, a CM galois extension of $\Q$
containing the \nth{p} roots of unity and we assume that, for sufficiently
large $n$, the \nth{p^n} roots are not contained in $\K_{n-1}$, but they are
in $\K_n$. Let $\KL \subset \KH_{\infty}$ be a subextension with galois group
$\Gal(\KL/\K_{\infty}) = \varphi(M) \vert_{\KL}$, with $M \subset \rg{A}$ a
$\Lambda$-submodule which is $\Z_p$-free. Let $\KL_n = \KL \cap \KH_n$ be the
finite levels of this extension and let $z \in \Z$ be such that $\exp(M_n) =
p^{n+z}$ in accordance with \rf{ordinc}. If $z < 0$, we may take $z = \max(z,
0)$. We define $\KL'_n = \KL_n \cdot \K_{n+z}$, so that $\KL'_n/\K_{n+z}$ is a
Kummer extension and let $R_n = \Rad(\KL'_n/\K_{n+z}$ and $B_n \cong
R_n^{p^{n+z}} \subset \K_n^{\times}/(\K_n^{\times})^{p^{n+z}}$. Then
\ref{ordinc} implies, by duality, that $R_{n+1}^p = R_n$, for $n > n_0$; the
radicals form a norm coherent sequence both with respect to the dual norm
$N_{m,n}^*$ and to the simpler $p$-map. Since $\KL = \cup_n \KL'_n$, we may
define $\Rad(\KL/\K_{\infty}) = \varprojlim_n R_n$.  The construction holds in
full generality for infinite abelian extensions of some field containing $\Q[
\mu_{p^{\infty}} ]$, with galois groups which are $\Z_p$-free
$\Lambda$-modules and projective limits of finite abelian $p$-groups. But we
shall not load notation here for presenting the details. Also, the extension
$\KL$ needs not be unramified, and we shall apply the same construction below
for $p$-ramified extensions.

We gather the above mentioned facts for future reference in
\begin{lemma}
\label{rads}
Let $z  \in \N$ be such that $\ord(a_n) \leq p^{n+z}$ for all $n$ and
$\K'_n = \K_{n+z}, \KL'_n = \KL_n \cdot \K_{n+z}$. Then $\KL'_n/\K'_n$ are
abelian Kummer extensions with galois groups $\Gal(\KL'_n/\K_{n+z}) \cong
\varphi(M_n)$, galois over $\K$ and with radicals $R_n = \Rad(\KL'_n/\K_{n+z})
\cong (\Gal(\KL'_n/\K_{n+z}))^{\bullet}$, as $\Lambda$-modules.  Moreover, if
$M = \Lambda c$ is a cyclic $\Lambda$-module, then there is a
$\nu^*_{n+1,n}$-compatible system of generators $\rho_n \in R_n$ such that
$R_n^{\bullet} = \Lambda \rho_n$ and, for $n$ sufficiently large,
$\rho_{n+1}^p = \rho_n$.  The system $R_n$ is projective and the limit is $R =
\varprojlim_n R_n$. We define
\[ \K_{\infty}[ R ] = \cup_n \K_{n+z}[ R_n ] = \KL. \]
\end{lemma}
Note that the extension by the projective limit of the radicals $R$ is a
convention, the natural structure would be here an injective limit. However,
this convention is useful for treating radicals of infinite extensions as
stiff objects, dual to the galois group which is a projective
limit. Alternatively, one can of course restrict to the consideration of the
finite levels.

The order reversal is a phenomenon reminiscent of the inverse galois
correspondence; if $M$ is cyclic annihilated by $f^n(T)$, with $f$ a
distinguished polynomial, then there is an inverse correspondence between the
$f$-submodules of $M$ and the $f^*$ submodules of the radical $R$. The result
is the following:
\begin{lemma}
\label{revers}
Let $f \in \Z_p[ T ]$ be a distinguished polynomial and $a \in A^- \setminus
A^p$ have characteristic polynomial $f^m$ for $m > 1$ and let $\id{A}_n =
\Lambda a_n, \id{A} = \Lambda$. Assume that $\KL \subset \KH_{\infty}$ has
galois group $\Delta = \Gal(\KL/\K_{\infty}) \cong \id{A}$ and let $R =
\Rad(\KL/\K_{\infty})$. At finite levels, we have $\Gal(\KL_n/\K_n) \cong
\id{A}_n$ and $R_n = \Rad(\KL'_n/\K_{n+z})$, with $R_n = \Lambda \rho_n$. Then
\begin{eqnarray}
\label{ordrevers}
\lan \varphi(a_n)^{f^k}, \rho_n^{(f^*)^{j}} \ran_{\tiny{\KL'_n/\K_{n+z}}} = 1 \quad \hbox{ for $k + j  \geq m$}.
\end{eqnarray} 
\end{lemma}
\begin{proof}
Let $g = \varphi(a_n) \in \Delta_n$ be a generator and $\rho \in R_n$ generate
the radical. The equivariance of Kummer pairing implies
\begin{eqnarray*}
\lan g^{f^k}, \rho^{(f^*)^{j}} \ran_{\KL'_n/\K_{n+z}}  =  
\lan g, \rho^{(f^*)^{j+k}} \ran
 =  \lan g^{f^{j+k}}, \rho \ran.
\end{eqnarray*}
By hypothesis, $a_n^{f^m} = 1$, and using also duality, $g^{f^m} =
\rho^{(f^*)^m} = 1$. Therefore, the Kummer pairing is trivial for $k+j \geq
m$, which confirms \rf{ordrevers} and completes the proof.
\end{proof}
It will be useful to give a translation of \rf{ordrevers} in terms of
projective limits: under the same premises like above, writing $\rho =
\varprojlim_n \rho_n$ for a generator of the radical $R =
\Rad(\KL/\K_{\infty})$, we have
\begin{eqnarray}
\label{infordrevers}
\lan \varphi(a)^{f^k}, \rho^{(f^*)^{j}} \ran_{\tiny{\KL/\K}} = 1 \quad \hbox{ for $k + j  \geq m$}.
\end{eqnarray} 
We shall also use the following simple result:
\begin{lemma}
\label{t2}
Let $\K$ be a CM galois extension of $\Q$ and suppose that $(\rg{A}')^-(T)
\neq 0$. Then $\ord_T(\rg{A}^-(T)) > 1$.
\end{lemma}
\begin{proof}
  Assuming that $(\rg{A}')^-(T) \neq 0$, there is some $a = (a_n)_{n
    \in \N} \in \rg{A}^-$ with image $a' \in (\rg{A}')^-[ T ]$. We
  show that $\ord_T(a) = 2$.  Let $\eu{Q}_n \in a_n$ be a prime and
  $n$ sufficiently large; then $\ord(a_n) = p^{n+z}$ for some $z \in
  \Z$ depending only on $a$ and not on $n$. Let $(\alpha_0) =
  \eu{Q}^{p^{n+z}}$ and $\alpha = \alpha_0/\overline{\alpha_0}$; since
  $a' \in (\rg{A}')^-[ T ]$ it also follows that $a_n^T \in \rg{B}^-$
  and thus $\eu{Q}^T = \eu{R}_n$ wit $b_n := [ \eu{R}_n ] \in
  \rg{B}_n$. If $b_n \neq 1$, then $\ord_T(a) = 1 + \ord_T(a') = 2$,
  and we are done.

  We thus assume that $b_n = 1$ and draw a contradiction. In this case
  $\eu{R}_n^{1-\jmath} = (\rho_n)$ is a $p$-unit and $(\alpha^T) =
  (\rho_n^{p^{n+z}})$, so
\[ \alpha^T = \delta \rho_n^{p^{n+z}}, \quad \delta \in \mu_{p^n}. \]
Taking the norm $N = N_{\K_n/\K}$ we obtain $1 = N(\delta)
N(\rho_n)^{p^{n+z}}$. The unit $N(\delta) \in \mu(\K) = < \zeta_{p_k} >$ -- we
must allow here, in general, that $\K$ contains the \nth{p^k} roots of unity,
for some maximal $k > 0$. It follows that $\rho_1 := N(\rho_n)$ verifies
$\rho_1^{p^{n+z}} = \delta_1$, and since $\delta_1 \not \in E(\K)^{p^{k+1}}$,
it follows that $\rho_1^{p^k} = \pm 1$ and by Hilbert 90 we deduce that
$\rho_n^{p^k} = \pm x^T, x \in \K_n^{\times}$. In terms of ideals, we have
then 
\begin{eqnarray*}
\eu{Q}^{(1-\jmath) T p^{n+z}} & = & (\alpha^T) = (x^{T p^{n+z-k}}), \quad
\hbox{hence} \\
\left(\eu{Q}^{(1-\jmath) p^k}/(x)\right)^{T p^{n+z-k}} & = & (1) \quad
\Rightarrow (\eu{Q}^{(1-\jmath) p^k}/(x))^T = (1).
\end{eqnarray*}   
But $\eu{Q}$ is by definition not a ramified prime, so the above implies that
$a_n$ has order bounded by $p^k$, which is impossible since $a_n \in
A_n^-$. This contradiction confirms the claim and completes the proof of the
lemma. 
\end{proof}

\subsection{Units and the radical of $\Omega$}
The extension $\Omega/\K$ is an infinite extension and $\zprk(\Gal(\Omega/\K))
= \id{D}(\K) + r_2(\K_n) + 1$. Here $r_2(\K_n)$ is the number of pairs of
conjugate complex embedding and the $1$ stands for the extension
$\K_{\infty}/\K$. Let $\wp \subset \K$ be a prime above $p$, let $D(\wp)
\subset \Delta$ be its decomposition group and $C = \Delta/D(\wp)$ be a set of
coset representatives in $\Delta$. We let $s = | C |$ be the number of primes
above $p$ in $\K$. Moreover
\[ \zprk\left(\Gal((\Omega^- /\K_{\infty})\right) =
  r_2(\K_n) . \] It is a folklore fact, which we shall prove constructively
  below, that the \textit{regular} part $r_2(\K_n)$ in the above rank stems
  from $\Omega^- \subset \Omega_{E}$, where $\Omega_E = \cup_n \K_n[ E_n\pn
  ]$. The radical is described precisely by:
\begin{lemma}
\label{r2}
Notations being like above, we define for $n > 1$: $\id{E}'_n = \{
e^{\nu_{n,1}^*} : e \in E_n \}$ and $\id{E}_n = \id{E}'_n \cdot (E_n)^{p^n}$.
Then
\begin{eqnarray}
\label{ommin}
\Omega^- = \KH_1 \cdot \cup_n \K_n[ \id{E}_n^{1/p^n} ]
\times \T_1,
\end{eqnarray} 
where $\T_1/\K_1$ is an extension which shall be described in the proof. It
has group $\Gal(\T_1/\K_1) \cong (\ZM{p})^{s-1}$.
\end{lemma}
\begin{proof}
  We show that the subgroups $\id{E}_m$ give an explicite construction
  of $\Omega^-$, as radicals. The proof uses reflection,
  class field theory and some technical, but strait forward
  estimations of ranks.

Let $U = \id{O}(\K \otimes_{\Q} \Q_p)$ and $U^{(1)}$ be the units congruent to
one modulo an uniformizor in each completion of $\K$ at a prime above $p$. The
global units $E_1 = E(\K_1)$ embed diagonally in $U$ and we denote by
$\overline{E}$ the completion of this embedding, raised to some power coprime
to $p$, so that $\overline{E} \subset U^{(1)}$.  A classical result from class
field theory \cite{La} p. 140, says that
\[ \Gal(\Omega/\KH_1) \cong U^{(1)}/\overline{E} . \] 
Since $(U^{(1)})^- \cap \overline{E} = \mu_{p}$, it follows that
$\Gal(\Omega^-/\KH_1^-) = (U^{(1)})^- \times \id{T}(U^-)/\mu_{p}$, where the
torsion part $\id{T}(U^-) = \prod_{\nu \in C} \mu_{p}$ is\footnote{We have
assumed for simplicity that $\K$ does not contain the \nth{p^2} roots of
unity. The construction can be easily generalized to the case when $\K$
contains the \nth{p^k} but not the \nth{p^{k+1}} roots of unity.} the product
of the images of the $p$ - th roots of unity in the single completions,
factored by the diagonal embedding of the global units.

For the proof, we need to verify that ranks are equal on both sides of
\rf{ommin}.  Let $\pi_{\nu} \in \K_n$ be a list of integers such that
$(\pi_{\nu}) = \wp^{\nu h}$ for $h$ the order of the class of $\wp^{\nu}$ in
the ideal class group $\id{C}(\K)$. Then we identify immediately $\T_1 =
\prod_{\nu \in \C} \K[ \pi_{\nu}^{1/p} ]$ as a $p$ - ramified extension with
group $\Gal(\T_1/\K) = \id{T}(U^-)/\mu_{p} \subset\Gal(\Omega^-/\KH^-)$.

A straight forward computation in the group ring yields that $T^* x \equiv 0
\bmod (\omega_n, p^n) \Lambda$ iff $x \in \nu^*_{n,1} \Lambda$. On the other
hand, suppose that $x \in \rad(\Omega^-/\K_n) \cap E_n$; note that here the
extensions can be defined as Kummer extensions of exact exponent $p^n$, so
there is no need of an index shift as in the case of the unramified extensions
treated above. This observation and Kummer theory imply that $x^{T^*} \in
E^{p^n}_n$, and thus $x \in \id{E}_n$. We denote as usual $\Omega_E = \cup_n
\K_n[ E_n\pn ]$. We found that $\cup_m \K_m[ \id{E}_m^{1/p^m} ] = \Omega^-
\cap \Omega_E$; by comparing ranks, we see that if $\Omega^- \neq \T_n \cdot
\KH_1 \cdot (\Omega^- \cap \Omega_E)$, then there is an extension $\Omega^-
\supset \Omega'' \supsetneq (\Omega^- \cap \Omega_E)$, such that
\[ \zprk(\Gal(\Omega''/\K_{\infty})) = r_2(\K) = \zprk(\Gal(\Omega^-
\cap \Omega_E)). \] Since $\Omega_E \subset \overline{\Omega}$, where
$\overline{\Omega}$ is the maximal $p$-abelian $p$-ramified extension of
$\K_{\infty}$, it follows that $\Gal((\Omega^- \cap \Omega_E)/\K_{\infty})$ is
a factor of $\Gal(\Omega^-/\K_{\infty})$ and also of
$\Gal(\Omega''/\K_{\infty})$.

The index $[ \Gal(\Omega'' : \K_{\infty}) : \Gal((\Omega^- \cap
  \Omega_E)/\K_{\infty}) ] < \infty$ and since $\Gal(\Omega''/\K_{\infty})$
is a free $\Z_p$ - module and thus has no finite compact subgroups, it follows
from infinite galois theory that $\Omega'' = \Omega^- \cap \Omega_E$,
which completes the proof.
\end{proof}
We note that for $\Omega_n \supset \K_n$, the maximal $p$-abelian $p$-ramified
extension of $\K_n$, the same arguments lead to a proof of
\begin{eqnarray}
\label{omen}
\Omega^-_n =  \cup_{m \geq n} \K_m\left[ E(\K_m)^{N_{m,n}^*/p^m}\right].
\end{eqnarray}


\subsection{Construction of auxiliary extension and order reversal}
On minus parts we have $\zprk(\Omega^-/\K^-) = r_2 + 1$ and the rank
$\zprk(\Omega^-/\overline{\KH}^-_1) = r_2$ does not depend on
Leopoldt's conjecture. We let $G = \Gal(\KH_{\infty}/\K)$ and $X =
\varphi(\rg{A}) = \Gal(\KH_{\infty}/\K_{\infty})$, following the
notation in \cite{Wa}, Lemma 13.15. The commutator is $G' = T X$ and
the fixed field $\KL = \KH_{\infty}^{T X}$ is herewith the maximal
abelian extension of $\K$ contained in $\KH_{\infty}$. From the
definition of $\Omega$, it follows that $\KL = \Omega \cap
\KH_{\infty}$ (see also \cite{Iw}, p. 257). Consequently
$\Gal(\KL/\overline{\KH_1}) \cong X/T X$. Let $F(T) = T^m G(T)$ be the
annihilator polynomial of $p^M \rg{A}$, with $p^M$ an annihilator of
the $\Z_p$-torsion (finite and infinite) of $\rg{A}$. If
$\rg{A}^{\circ}$ is this $\Z_p$-torsion, then $\rg{A} \sim \rg{A}(T)
+ \rg{A}(G(T)) + \rg{A}^{\circ}$.

From the exact sequences 
\begin{eqnarray*}
\xymatrix{ 0  \ar@{->}[r] & K_1 \ar@{->}[r]\ar@{->}[d] & p^{\mu} \rg{A}^- \ar@{->}[r]\ar@{->}[d] & p^{\mu} \rg{A}^-(T) + \rg{A}^-(G)  \ar@{->}[r]\ar@{->}[d] & K_2 \ar@{->}[r]\ar@{->}[d] & 0 \\
0  \ar@{->}[r] & 0  \ar@{->}[r] & p^{M} \rg{A}^-  \ar@{->}[r] & p^M \rg{A}^-(T) \oplus \rg{A}^-(G)  \ar@{->}[r] & 0  \ar@{->}[r] & 0 
}
\end{eqnarray*}
in which $M \geq \mu$ is such that annihilates the finite kernel and cokernel
$K_1, K_2$ and the vertical arrows are multiplication by $p^{M-\mu}$, we see
that it is possible to construct a submodule of $\rg{A}^-$ which is a direct
sum of $G$ and $T$-parts. We may choose $M$ sufficiently large, so that the
following conditions also hold: $p^M \rg{A}^-(T)$ is a direct sum of cyclic
$\Lambda$-modules and if a prime above $p$ is inert in some
$\Z_p$-subextension of $\KH_{\infty}/\KH_{\infty}^{p^M \varphi(\rg{A}^-)}$,
then it is totally inert. Let $\tilde{\K} = \KH_{\infty}^{p^M
\varphi(\rg{A}^-)}$ for some $M$ large enough to verify all the above
conditions. Let $\K_T = \KH_{\infty}^{p^M \rg{A}^-(G)}$; by construction,
$\tilde{X}_T := \Gal(\K_T/\tilde{\K}) \sim \rg{A}^-(T)$ and it is a direct sum
of cyclic $\Lambda$-modules. Let $a_1, a_2, \ldots, a_t \in p^M \rg{A}^-(T) =
\varphi^{-1}(\tilde{X}_T)$ be such that
\begin{eqnarray}
\label{atdirsum}
\tilde{X}_T = \bigoplus_{j=1}^t \varphi\left(\Lambda a_i\right).
\end{eqnarray}
From the definition $\K_B^- = \Omega^- \cap \KH_{\infty} \subset \K_T$ and
Lemma \ref{t2} implies that $\Gal(\KH_{\infty}/\K_B^-) \sim \tilde{X}_T[ T
]$. Let now $a \in p^M \rg{A}^-(T) \setminus (p, T) p^M \rg{A}^-(T)$ - for
instance $a = a_1$ and let $\id{A} = \Lambda a$ while $\id{C} \subset p^M
\rg{A}^-(T)$ is a $\Lambda$-module with $\id{A} \oplus \id{C} = p^M
\rg{A}^-(T)$. We assume that $m = \ord_T(a)$ and let $b = T^{m-1} a \in
\rg{A}^-[ T ]$. We define $\K_a = \K_T^{\varphi(\id{C})}$, an extension with
$\Gal(\K_a/\tilde{\K}) \cong \id{A}$. At finite levels we let $\K_{a,n} :=
\K_a \cap \KH_n$ and let $z$ be a positive integer such that
$\K'_{a,n}:=\K_{n+z} \K_{a,n}$ is a Kummer extension of $\K'_n := \K_{n+z}$,
for all sufficiently large $n$ -- we may assume that $M$ is chosen such that
the condition $n > M$ suffices. The duals of the galois groups
$\varphi(\id{A}_n)$ are radicals $R_n = \Rad(\K_{a,n}/\tilde{\K})$, which are
cyclic $\Lambda$-modules too (see also the following section for a detailed
discussion of radicals), under the action of $\Lambda$, twisted by the
Iwasawa involution. We let $\rho_n \in R_n$ be generators which are dual to
$a_n$ and form a norm coherent sequence with respect to the $p$-map, as was
shown above, since $n > M > n_0$; by construction, $\rho_n^{p^{n+z}} \in
\K'_n$. We gather the details of this construction in
\begin{lemma}
\label{at}
Notations being like above, there is an integer $M > 0$, such that the
following hold:
\begin{itemize}
 \item[ 1. ] The extension $\tilde{\K} := \KH_{\infty}^{p^M X}$ has group
 $\tilde{X} := \Gal(\KH_{\infty}/\tilde{\K}) = \tilde{X}(T) \oplus
 \tilde{X}(G)$ below $\KH_{\infty}$.
\item[ 2. ] The extension $\K_T := \KH_{\infty}^{\tilde{X}(G)}$ has group
$\tilde{X}_T = \bigoplus_{i=1}^t \Lambda \varphi(a_i)$.
\item[ 3. ] For $a \in p^M \rg{A}^-(T) \setminus (p, T) p^M \rg{A}^-(T)$ we
define $\id{A} = \Lambda a$ and let $\id{C} \subset p^M \rg{A}^-(T)$ be a
direct complement. We define $\K_a = \KH_T^{\varphi(\id{C})}$, so
$\Gal(\K_a/\tilde{\K}) = \varphi(\id{A})$ and let $\K_{a,n} = \K_a \cap
\KH_n$.
\item[ 4. ] There is a positive integer $z$ such that for all $n > M$, 
\[ \K'_{a,n} = \K_{n+z} \cdot \K_{a,n} \subset \KH_{n+z} \]
 is a Kummer extension of $\K'_n := \K_{n+z}$.
\item[ 5. ] For $\K_B^- = \Omega^- \cap \KH_{\infty}$ we have $\K_B^- \subset
\K_T$ and $\Gal(\KH_{\infty}/\K_B^-) \sim \tilde{X}_T[ T ]$.
\item[ 6. ] The radical $R_n = \Rad(\K_{a,n}/\tilde{\K}) \cong
\id{A}_n^{\bullet}$ and we let $\rho_n \in R_n$ generate this radical as a
$\Lambda^*$-cyclic module, so that $\rho_n^{(T^*)^i}, i = 0, 1, \ldots, m-1$
form a dual base to the base $a_n^{T^i}, i = 0, 1, \ldots, m-1$ of
$\id{A}_n$. We have $\rho_n^{p^{n+z}} \in \K'_n$.
\end{itemize}
\end{lemma}

We may apply the order reversal to the finite Kummer extensions
$\K_{a,n}/\tilde{\K}_n$ defined in Lemma \ref{at}. In the notation of this
lemma, we assume that $m = \ord_T(a) > 1$. We deduce from \ref{ordrevers} that 
\begin{eqnarray}
\label{eplicitrev}
\lan \varphi(a_n)^{T^i}, \rho_n^{(T^*)^{m-1-i}} \ran_{\K_{a,n}/\tilde{\K}} =
  \zeta_{p^v}, \\
& & \nonumber \quad v \geq n + z - M, \ i = 0, 1, \ldots, m-1 .
\end{eqnarray}
This fact is a direct consequence of \rf{ordrevers} for $i=0$ and it follows
by induction on $i$, using the following fact. Let $\F_i = \tilde{\K}[
\rho_n^{(T^*)^{m-1-i}} ]$; then $\overline{\F}_i = \prod_{j=0}^i \F_i$ are
galois extensions of $\tilde{\K}_1$ and in particular their galois groups are
$\Lambda$-modules. In particular, $\Gal(\F_{m-1}/\tilde{\K}) \cong
\id{A}_n^{T^{m-1}} = \id{A}_n[ T ]$. From Lemma \ref{t2} we know that
$\id{A}_n[ T ] \subset \rg{B}_n^-$, so at least one prime $\eu{p} \subset
\tilde{\K}$ above $p$ is inert in $\F_{m-1}$, and the choice of $M$ in Lemma
\ref{at} implies that it is totally inert in $\F_{m-1}/\tilde{\K}_n$. Let $\wp
\subset \K$ be a prime below $\eu{p}$. It follows in addition
$\overline{\F}_{m-2} \subset \KH'_{\infty} \cdot \tilde{\K}$ and all the
primes above $p$ are split in $\overline{\F}_{m-2}$: this is because
\[ \Gal(\overline{\F}_{m-2}/\tilde{\K}) \cong \id{A}_n/\id{A}_n[ T ] = 
\id{A}_n/(\id{A}_n \cap \rg{B}_n) \subset A'_n . \]

Let now $\K_{a}$ be like above and $\K_b = \Omega^- \cap \K_a$, so
$\K_b/\tilde{\K}$ is a $\Z_p$-extension. Moreover, we assume that $\K_b \not
\subset \KH'_{\infty}$, so not all primes above $p$ are totally split. By
choice of $M$, we may assume that there is at least on prime $\wp \subset \K$
above $p$, such that the primes $\tilde{K} \supset \eu{p} \supset \wp$ are
inert in $\K_b$. By the construction of $\Omega^-$ in the previous section, we
have $T^* \Rad(\K_b/\tilde{\K}) = 0$. The order reversal lemma implies then
that $\Gal(\K_b/\tilde{\K}) \cong \id{A}/(T \id{A})$. Assuming now that $m =
\ord_T(a) \geq 1$, the Lemma \ref{t2} implies that $T^{m-1} a \in \rg{B}^-$
and the subextension of $\K_a$ which does not split all the primes above $p$
is the fixed field of $T^{m-1} \id{A}$; but then order reversal requires that
$\Rad(\K_b/\tilde{\K})$ is cyclic, generated by $\rho$, which is at the same
time a generator of $\Rad(\K_a/\tilde{\K})$ as a $\Lambda$-module. Since we
have seen that $T^* \Rad(\K_b/\tilde{\K}) = 0$, we conclude that $T^*
\Rad(\K_a/\tilde{\K}) = 0$, and by duality, $T a = 0$. This holds for all $a
\in p^M \rg{A}^-(T)$, so we have proved:
\begin{lemma}
\label{gaussum}
Let $\KH_B^- = \Omega^- \cap \KH_{\infty}$. If $[ \K_B^- \cap \KH'_{\infty} :
\K_{\infty} ] < \infty$, then $\rg{A}^-(T) = \rg{B}^-$. 
\end{lemma}

\subsection{The contribution of class field theory}
We need to develop more details from local class field theory in order to
understand the extension $\KH_B^- = \Omega^- \cap \KH_{\infty}$. This is an
unramified extension of $\K_{\infty}$ which is abelian over $\KH_1$. We wish
to determine the $\Z_p$-rank of this group and decide whether the extensions
in $\KH_B^-$ split the primes above $p$ or not.

Let $\wp \subset \K$ be a prime over $p$ and $\wp^+ \subset \K^+$ be the real
prime below it. If $\wp^+$ is not split in $\K/\K^+$, then $\rg{B}^- = \{ 1
\}$ and it is also known that $(\rg{A}')^-(T) = \{ 1 \}$ in this case -- this
follows also from the Lemma \ref{t2}. The case of interest is thus when
$\wp$ is split in $\K/\K^+$. Let $D(\wp) \subset \Delta$ and $C, s$ be defined
like above and let $\eu{p} \subset \Omega$ be a prime above $\wp$. 

Local class theory provides the isomorphism $\Gal(\Omega/\KH_1) \cong
U^{(1)}/\overline{E}$ via the global Artin symbol (e.g. \cite{La}, ).  We
have the canonic, continuous embedding
\[ \K \hookrightarrow \K \otimes_{\Q} \Q_p \cong \prod_{\nu \in C} \K_{\nu
  \wp},  \] 
and $U^{(1)} = \prod_{\nu \in C} U^{(1)}_{\nu \wp}$, where $U^{(1)}_{\eu{p}}$
are the one-units in the completion at the prime $\eu{p}$. The ring $U^{(1)}$
is a galois algebra and $\Delta = \Gal(\K/\Q) \hookrightarrow
\Gal(U^{(1)}/\Q_p)$. Thus complex conjugation acts on $U^{(1)}$ via the
embedding of $\K$ and if $u \in U^{(1)}$ has $\iota_{\wp}(u) = x,
\iota_{\overline{\wp}}(u) = y$, then $\jmath u$ verifies
\[ \iota_{\wp}(\jmath u) =  \overline{y},  \quad \iota_{\overline{\wp}}(\jmath
u) = \overline{x}. \]
Moreover, $u \in U^-$ iff $u = v^{1-\jmath}, v \in U$. Thus, if $\iota_{\wp}(v)
= v_1$ and $\iota_{\overline{\wp}}(v) = v_2$, then 
\begin{eqnarray}
\label{comcon}
\iota_{\wp}(u)  =  v_1/\overline{v_2}, \quad
\iota_{\overline{\wp}}(u) =  v_2/\overline{v_1} = 1/\overline{\iota_{\wp}(u)}.
\end{eqnarray}
One can analyze $U^+$ in a similar way. Note that $\Z_p$ embeds diagonally in
$U^+$; this is the preimage of $\Gal(\K_{\infty}/\K)$, under the global Artin
symbol. 

Next we shall construct by means of the Artin map some subextension of
$\Omega^-$ which are defined uniquely by some pair of complex conjugate primes
$\wp, \overline{\wp} \supset (p)$ and intersect $\KH_{\infty}$ is a
$\Z_p$-extension. Since $U^{(1)}$ is an algebra, there exists for each pair of
conjugate primes $\wp, \overline{\wp}$ with fixed primes $\eu{P}, \eu{P}^{\tj}
\subset \Omega$ above $(\wp, \overline{\wp})$, a subalgebra
\begin{eqnarray}
\label{vdef}
 V_{\wp} = \{ u \in U^{(1)} \ : \
\iota_{\eu{P}}(u) = 1/\overline{\iota_{\eu{P}^{\jmath}} (u)}; \iota_{\nu \wp}
= 1, \ \forall \nu \in C \setminus \{1, \jmath\} \} .
\end{eqnarray}
Accordingly, there is an extension $\M_{\wp} \subset \Omega^-$ such that 
\[ \varphi^{-1} ( \Gal(\M_{\wp}/\K_{\infty}) ) = V_{\wp} . \]
By construction, all the primes above $p$ above $\wp, \overline{\wp}$ are
totally split in $\M_{\wp}$. Since $\Gal(\K_{\wp}/\Q_p) = D(\wp)$ and
$U_{\wp}$ is a pseudocyclic $\Z_p$-module, pseudoisomorphic to $\Z_p[ D(\wp)
]$ (e.g. \cite{La}, p. 140-41), it follows that there is exactly one
$\Z_p$-subextension $\U_{\wp} \subset \M_{\wp}$ with galois group fixed by the
augmentation of $D(\wp)$. Since the augmentation and the norm yield a direct
sum decomposition of $\Z_p[ D(\wp) ]$, this extension and its galois group are
canonic -- up to possible finite quotients. Locally, the completion of
$\rg{U}/\Q_p$ of $\U_{\wp}$ at the primes above $\wp$ is a $\Z_p$-extension of
$\Q_p$, since its galois group is fixed $D(\wp)$. It follows by a usual
argument that $\U_{\wp}/\K_{\infty}$ is unramified at all primes above $p$, so
$U_{\wp} \subset \KH$. One has by construction that $\U_{\wp}^- \subset
\Omega^-$, so we have proved:
\begin{lemma}
\label{atwp}
Let $\K$ be a CM extension like above and assume that the primes $\wp^+
\subset \K^+$ split in $\K/\K^+$. For each prime $\wp \subset \K$ there is a
canonic (up to finite subextensions) $\Z_p$-extension $\U_{\wp} \subset
\Omega^- \cap \KH_{\infty}$ such that $\Gal(\U_{\wp}/\K_{\infty}) =
\varphi\left(V_{\wp}^{\eu{A}(\Z_p[ D_{\wp} ])}\right)$, where $\eu{A}(\Z_p[
D_{\wp} ]$ is the augmentation ideal of this group ring and $V_{\wp}$ is
defined by \rf{vdef}. In particular, $\Omega^-$ contains exactly $s' = | C |/2$
unramified extensions.
\end{lemma}
Our initial question boils down to the following: is $\U_{\wp} \subset
\KH'_{\infty}$?

The following example perfectly illustrates the question: 
\begin{example}
  Let $\K/\Q$ be an imaginary quadratic extension of $\Q$ in which $p$ is
  split. Then $U^{(1)}(\K) = (\Z_p^{(1)})^2$ and $\Omega = \K_{\infty} \cdot
  \Omega^-$ is the product of two $\Z_p$-cyclotomic extensions; we may assume
  that $\KH_1 = \K$, so $\Gal(\Omega/\K) = \varphi(U^{(1)}(\K))$. One may take
  the second $\Z_p$-extension in $\Omega$ also as being the anticyclotomic
  extension. In analyzing a similar example, Greenberg makes in \cite{Gr} the
  following simple observation: since $\Q_p$ has only two $\Z_p$-extensions
  and $\K_{\infty}$ contains the cyclotomic ramified one, it remains that,
  locally $\Omega^-/\K_{\infty}$ is either trivial or an unramified
  $\Z_p$-extension. In both cases, $\Omega^- \subset \KH_{\infty}$ is a
  global, totally unramified $\Z_p$-extension -- we have used the same
  argument above in showing that $\U_{\wp}/\K_{\infty}$ is unramified. The
  remark settles the question of ramification, but does not address the
  question of our concern, namely splitting. However, in this case we know
  more. In the paper \cite{Gr1} published by Greenberg in the same year, he
  proves that for abelian extensions of $\Q$, thus in particular for quadratic
  ones, $(\rg{A}')^-(T) = \{ 1 \}$. Therefore in this example, $\Omega^-$
  cannot possibly split the primes above $p$.

How can this fact be explained by class field theory? 
\end{example}
We give here a proof of Greenberg's theorem \cite{Gr1} for imaginary quadratic
extensions, and thus an answer to the question raised in the last example; we
use the notations introduced there:
\begin{proof}
We shall write $\KL = \K_{\infty} \cdot \KH_1$; we have seen
above that $\Omega/\KL$ must be an unramified extension. Let $\eu{P} \in \Omega$ be a prime above $\wp$, let $\tj \in \Gal(\Omega/\KH_1)$ be a lift of complex
conjugation and let $\tau \in \Gal(\Omega/\KH_1)$ be a generator of the
inertia group $I(\eu{P})$: since $\Omega_{\eu{P}}/\K_{\wp}$ is a product of $\Z_p$-extensions of $\Q_p$ and $\Q_p$ has no two independent ramified $\Z_p$-extensions, it follows that $I(\eu{P}) \cong \Z_p$ is cyclic, so $\tau$ can be chosen as a topological generator. Then $\tau^{\jmath} = \jmath \cdot \tau
\cdot \jmath$ generates $I(\eu{P}^{\tj}) \cong \Z_p$. 
Iwasawa's argument used in the
proof of Thereom \ref{iw6} holds also for $\Omega/\KH_1$: there is a
class $a \in A_n$ with $\tau^{\jmath} = \tau \varphi(a)$, where the Artin
symbol refers to the unramified extension $\Omega/\KL$. Thus
\[ \jmath \cdot \tau \cdot \jmath \cdot \tau^{-1} = \tau^{\jmath-1} = \varphi(a). \]
The inertia groups $I(\eu{P}) \neq I(\eu{P}^{\tj})$: otherwise, their common fixed field would be an unramified $\Z_p$-extension of the finite galois field $\KH_1/\Q$, which is impossible: thus $\tau^{\jmath-1} = \varphi(a) \neq 1$ generates a group isomorphic to $\Z_p$. Let now $\eu{p} =
\eu{P} \cap \KL$; the primes $\eu{p}, \eu{p}^{\tj}$ are unramified in
$\Omega_n/\KL$, so $\tau$ restricts to an Artin symbol in this
extension. The previous identity implies
\[ \lchooses{\Omega/\KL}{a} =
\lchooses{\Omega/\KL}{\eu{p}^{\jmath-1}} ; \]
Since the Artin symbol is a class symbol, we conclude that the primes in the
coherent sequence of classes $b = [ \eu{p}^{\jmath-1} ] \in \rg{B}^-$ generate $\Gal(\Omega/\Omega^{\varphi(a)})$
and $a = b$, which completes the proof.
\end{proof}


\subsection{Proof of Theorems \ref{simram} and \ref{gross}}
We can turn the discussion of the example above into a proof of
Theorem \ref{simram} with its consequence, the Corollary \ref{gross}.
The proof generalizes the one given above for imaginary quadratic extensions,
by using the construction of the extensions $\U_{\wp}$ defined above. 
\begin{proof}
Let $\KL = \KH_1 \cdot \K_{\infty}$, like in the previous proof. Let $\wp \subset
\K$ be a prime above $p$ and $\U$ be the maximal unramified extension of $\KL$ contained in $\U_{\wp}$, the extension defined in Lemma \ref{atwp}, and let $\tj$ be a lift of complex conjugation to $\Gal(\U/\Q)$. Since $\Omega/\KH_1$ is abelian, the
extension $\U/\KH_1$ is also galois and abelian.

Let $\eu{P} \subset \U$ be a fixed prime above
$\wp$ and $\tj \in \Gal(\U/\KH_1)$ be a lift of complex conjugation. Consider
the inertia groups $I(\eu{P}), I(\eu{P}^{\tj}) \subset \Gal(\U/\KH_1)$ be the inertia groups of the two conjugate primes. Like in the example above, $\Gal(\U/\KH_1) \cong \Z_p^2$ and $\U_{\eu{P}}/\K_{\wp}$ is a product of at most two $\Z_p$ extensions of $\Q_p$. It follows that the inertia groups are isomorphic to $\Z_p$ and disinct: otherwise, there commone fixed field in $\U$ would be an uramified $\Z_p$-extension of $\KH_1$.

For $\nu \in C \setminus \{1, \jmath\}$, the primes above $\nu \wp$ are
totally split in $\U_{\wp}/\K_{\infty}$, so a fortiori in $\U$. 
Let $\tt \in \Gal(\U/\KH_1)$ generate the inertia group $I(\eu{P})$;
  then $\tt^{\tj} \in \Gal(\U/\KH_1)$ is a generator of $I(\eu{P}^{\tj})$. Since $\U/\KL$ is an unramified extension,
  there is an $a \in \rg{A}^-$ such that
  \[ \tt^{\jmath} = \jmath \tt \jmath = \lchooses{\U/\KL}{a} \cdot \tt .\] 
 Thus 
\begin{eqnarray}
\label{commut}
\varphi(a) = \jmath \tt \jmath \tt^{-1}.
\end{eqnarray}
Like in the previous proof, we let $\eu{p} = \eu{P} \cap \KL$ and note
that since $\eu{p}$ does not ramify in $\U/\KL$, the automorphism $\tt$
acts like the Artin symbol $\lchooses{\U/\KL}{\eu{p}}$. The relation
\rf{commut} implies:
\[ \lchooses{\U/\KL}{ a } = \lchooses{\U/\KL}{\eu{p}^{\jmath-1}}. \]
In particular, the primes in the coherent sequence of classes $b = [ \eu{p}^{\jmath-1} ] \in \rg{B}^-$ generate $\Gal(\U/\KL)$ and $\U$ does 
not split all the primes above $p$. This happens for all $\wp$ and by Lemma \ref{atwp} we have $\KH_B^-
= \prod_{\nu \in C/\{1,\jmath\}} U_{\nu \wp}$, so it is spanned by
$\Z_p$-extensions that do not split the primes above $p$ and consequently
\[ [ \KH'_{\infty} \cap \KH_B ] < \infty . \]
We may now apply Lemma \ref{gaussum} which implies that $\rg{A}^-(T) =
\rg{B}^-$. This completes the proof of Theorem \ref{simram}. The corollary
\ref{gross} is a direct consequence: since $\rg{A}^-(T) = \rg{A}^-[ T ] =
\rg{B}^-$, it follows directly from the definitions that $(\rg{A}')^-(T) = \{
1 \}$.
\end{proof}

\begin{remark}
\label{mu}
The above proof is intimately related to the case when $\K$ is CM and
$\K_{\infty}$ is the $\Z_p$-cyclotomic extension of $\K$. The methods cannot
be extended without additional ingredients to non CM fields, and certainly not
other $\Z_p$-extensions than the cyclotomic. In fact, Carroll and Kisilevsky
have given in \cite{CK} examples of $\Z_p$-extensions in which $\rg{A}'(T)
\neq \{ 1 \}$.

A useful consequence of the Theorem \ref{simram} is the fact that the
$\Z_p$-torsion of $X/ TX$ is finite. As a consequence, if $M = \rg{A}[
p^{\mu} ], \mu= \mu(\K)$, then $Y_1 \cap M^- \subset T X$. In
particular, if $a \in M^-$ has $a_1 = 1$, then $a \in T M^-$. We shall
give in a separate paper a proof of $\mu = 0$ for CM extensions, which
is based upon this remark. Note that the finite torsion of $X / T X $
is responsible for phenomena such as the one presented in the example
5. above.
\end{remark}

\section{Conclusions}
Iwasawa's Theorem 6 reveals distinctive properties of the main module $\rg{A}$
of Iwasawa Theory, and these are properties that are not shared by general
Noeterian $\Lambda$-torsion modules, although these are sometimes also called
``Iwasawa modules''. In this paper we have investigated some consequences of
this theorem in two directions. The first was motivated by previous results of
Fukuda: it is to be expected that the growth of specific cyclic
$\Lambda$-submodules which preserve the overall properties of $\rg{A}$ in
Iwasawa's Theorem, at a cyclic scale, will be constrained by some
obstructions. Our analysis has revealed some interesting phenomena, such as
\begin{itemize}
\item[1.] The growth in rank of the modules $\id{A}_n$ stops as soon as this
  rank is not maximal (i.e., in our case, $p^{n-1}$ for some $n$.
\item[2.] The growth in the exponent can occur at most twice before rank
  stabilization.
\item[3.] The most \textit{generous} rank increase is possible for regular
  flat module, when all the group $\id{A}_n$ have a fixed exponent and
  subexponent, until rank stabilization, and the exponent is already
  determined by $\id{A}_n$. It is an interesting fact that we did not
  encounter any example of such modules in the lists of Ernvall and
  Mets\"ankil\"a.
\end{itemize}
Although these obstruction are quite strong, there is no direct upper bound
either on ranks or on exponents that could be derived from these analysis.

Turning to infinite modules, we have analyzed in Chapter 4 the structure of
the complement of $T X$ in Iwasawa's module $Y_1^-$ in the case of CM
extentions. This was revealed to be $\rg{B}^-$, a fact which confirms the
conjecture of Gross-Kuz'min in this case. 

The methods introduced here suggest the interest in pursuing the investigation
of consequences of Iwasawa's Theorem. Interesting open topics are the
occurance of floating elements and their relation to the splitting in the
sequence \rf{bsplit} and possible intersections of $\Lambda$-maximal
modules. It conceivable that a better understanding of these facts may allow
to extend our methods to the study of arbitrary $\Lambda$-cyclic submodules of
$\rg{A}$. It will probably be also a matter of taste, to estimate whether the
detail of the work that such generalizations may require can be expected to be
compensated by sufficiently simple and structured final results.
\vspace*{0.3cm}

\textbf{Acknowledgment:} The material of this paper grew,
was simplified and matured over a long period of time, first as a central
target, then as a byproduct of deeper investigations in Iwasawa's theory. I
thank all my colleagues at the Mathematical Institute of the University of
G\"ottingen for their support in this time. I thank Victor Vuletescu who
actively helped the development of the ideas during the time of the Volkswagen
Foundation grant and Tobias Bembom and Gabriele Ranieri for their critical
reading and active help with the completion of this paper.

Last but not least, the questions discussed here were the first
building block of a series of Seminar lectures held together with
S. J. Patterson in the years 2007-08, the notes of which will appear
in the sequel in the series entitled SNOQIT: {\em Seminar Notes on
  Open Questions in Iwasawa Theory}. I thank Paddy for the
inspirational discussions we had since that time.
\bibliographystyle{abbrv} \bibliography{leq-genv2}

\end{document}